\documentclass
{amsart}

\usepackage{mathtools}
\usepackage {amsmath, amscd, 
amsthm, 
amssymb}
\usepackage{bm}
\usepackage{cancel}
\usepackage[margin=1in, centering]{geometry}
\numberwithin{equation}{section}
\newtheorem{thm}{Theorem}[section] 
\newtheorem{cor}[thm]{Corollary}
\newtheorem{lem}[thm]{Lemma}
\newtheorem{defn}[thm]{Definition}

\DeclareMathOperator{\Res}{Res}

\DeclareMathOperator{\glninf}{gl^{(n)}_\infty}

\newcommand{\inv}{^{-1}}
\newcommand{\abs}[1]{\lvert#1\rvert}

\newcommand{\fermion}[3]{{\prescript{}{#1}\psi_{(#2)}^{#3}}}

\begin{document}

\markboth{Addabbo and Bergvelt}
{Difference Hierarchies for  $nT$ $\tau$-functions}

\title{Difference Hierarchies for  $nT$ $\tau$-functions.
}

\author{Darlayne Addabbo
}

\address{Department of Mathematics, University of Notre Dame, Notre
  Dame, IN 46556, USA 
}

\author{Maarten Bergvelt}

\address{Department of Mathematics, University of Illinois, Urbana-Champaign, Il 61801, USA
}

\maketitle

\begin{abstract}
  We introduce hierarchies of difference equations (referred to as
  $nT$-systems) associated to the action of a (centrally extended,
  completed) infinite matrix group $GL_{\infty}^{(n)}$ on
  $n$-component fermionic Fock space. The solutions are given by
  matrix elements ($\tau$-functions) for this action. We show that the
  $\tau$-functions of type $nT$ satisfy bilinear equations of length
  $3,4,\dots,n+1$.  The $2T$-system is, after a change of variables,
  the usual $3$ term $T$-system of type $A$. 

  Restriction from $GL_{\infty}^{(n)}$ to a subgroup isomorphic to the loop group
  $LGL_{n}$, defines $nQ$-systems, studied earlier in
  \cite{taubirkhoff} by the present authors for $n=2,3$. 
\end{abstract}


\section{Introduction.}
\label{sec:intro}

There has recently been much interest in discrete integrable systems,
difference equations with many conserved quantities, in the context of,
for instance,  cluster algebras (see \cite{MR2452184},
\cite{MR2551179}, \cite{MR2773889}), or statistical physics (see
\cite{MR2773889}). In particular, the
$T$-system and its reduction, the $Q$-system has been studied in great
detail, with applications in representation theory and
combinatorics. It is known that the $T$- and $Q$-systems are discrete
Hirota equations, \cite{MR2773889}.

Experience in the general theory of integrable systems has shown
that Hirota equations do not occur in isolation: they are found to
appear in families, forming hierarchies of compatible equations.
Experience also shows that such integrable hierarchies are attached to
representation theoretic data. For instance, the KP hierarchy{ (see
e.g., \cite{MR1736222}, \cite{Sa:KPinfDymGr}, \cite{KaPe:LectinfwedgMKP})}
corresponds to the principal construction of the basic representation
of the infinite matrix group $GL_{\infty}$\footnote{ Actually, a
  central extension of a completion of $GL_{\infty}$ is what is needed
for KP. In this introduction, for expository purposes, we will usually ignore
these details.}. The KdV hierarchy, a
reduction of the KP hierarchy, is similarly related to the
loop group of $SL_{2}$, denoted $LSL_{2}$, which is a subgroup of $GL_{\infty}$.

It is then a natural question to ask whether the $T$-system and the
$Q$-system are part of more general hierarchies of difference
equations, and whether one can give a representation theoretic
construction of these hierarchies.\footnote{There are various types of $Q$-systems and $T$-systems, but in this paper we will restrict to $T$-systems and $Q$-systems of type $A$ and generalizations of these.} The first aim of this paper is to
show that the answer to both questions is positive. There are two
hierarchies of difference equations of which the $T$- and $Q$-systems are the
simplest members. Furthermore, these hierarchies are also attached to
representation theoretic data: for the $T$ and $Q$ system we have
(central extensions of) an
infinite matrix group $GL_{\infty}^{(2)}$ and a subgroup isomorphic to $LGL_{2}$
acting on 2-component fermionic Fock space, and these systems are related
in very much the same way as  the KP and KdV hierarchies: we propose the
slogan: ``$T$ is to $Q$ as KP is to KdV''.

The difference between the KP hierarchy and the $T$-system (both
connected to {infinite matrix groups}) lies in the
choice of bosonization. It is well known that different KP like
hierarchies can be constructed depending on the choice of a Heisenberg
subalgebra in the (centrally extended, completed) infinite matrix algebra, each
inequivalent Heisenberg algebra giving rise to a different
bosonization, {see for instance \cite{MR88c:17022},
  \cite{MR1155284}, \cite{MR97c:58061}.} In this paper we do something
slightly different; we choose a unipotent subgroup, instead of a
Heisenberg subalgebra.

Let us briefly sketch the construction of the KP hierarchy in order to
compare it to the construction in the present paper. For the KP
hierarchy, one can start with the fermionic Fock space $F$, the
semi-infinite wedge space of $\mathbb{C}^{\infty}$. On $F$ we have the actions of the matrix
algebra $gl_{\infty}$ and group $GL_{\infty}$, and the centrally
extended completion $a_{\infty}$ of $gl_{\infty}$.
There is a Heisenberg subalgebra in $a_{\infty}$, spanned by elements $\alpha(k)$ such
that
\[
[\alpha(k),\alpha(\ell)]=k\delta_{k,\ell}.
\]
On $F$, we also have a bilinear form $\langle\, ,\rangle$ such that
elementary wedges are orthonormal. Finally, fermionic Fock space has a
grading by charge, $F=\oplus_{k\in\mathbb{Z}}F_{k} $, where $F_{k}$ is
the charge $k$ subspace. Then $ F_{0}$ contains the vacuum vector
$v_{0}$ and one defines a map, called bosonization,
\[
\Phi\colon F_{0}\to \mathbb{C}[t_{1},t_{2},\dots],\quad \omega\mapsto
\langle v_{0},e^{\sum_{k=1}^{\infty}t_{k}\alpha(k)}\omega\rangle.
\]
If $\omega\in F_{0}$ is an element of the $GL_{\infty}$ orbit of
the vacuum, $\omega=gv_{0}$, we call $\Phi(\omega)$ the
$\tau$-function $\tau_{\omega}(t)$, where $ t=(t_{1},t_{2},\dots)$ are
the KP times. A remarkable fact is that
$\tau_{\omega}(t)$ solves an infinite collection of differential equations,
the KP hierarchy, \cite{MR86a:58093}, \cite{Kac:InfDimLA}.

This construction is an instance of the general philosophy of special
functions: interesting functions are matrix elements of an action of a
group\footnote{Or some other algebraic structure such as a quantum
  group, etc., see for instance \cite{MR1275729}, \cite{MR1302644}} $G$ on some representation $V$, with the variables in the
special function obtained by choosing coordinates on $G$ (or a
subgroup, say diagonal elements in a matrix group).  One can think of
the KP times $t=(t_{1},t_{2},\dots)$ as coordinates on a subgroup of
$GL_{\infty}$ (of matrices that are constant along diagonals).

The {starting point for this paper was the} observation that for
the $T$-system, one can choose, {instead of a Heisenberg subalgebra, a}
subgroup of an infinite matrix group $GL_{\infty}^{(2)}$
(acting on 2-component fermionic Fock space), consisting of certain
lower triangular matrices, with coordinates
$c_{k,\ell}, k,\ell\in\mathbb{Z}$. By more or less the
same construction as for the KP hierarchy, one defines
$\tau$-functions. Briefly, we start with an element of the lower triangular
subgroup of $GL_{\infty}^{(2)}$ mentioned above, lift it to an element $g$
of the central extension acting on the 2-component
fermionic Fock space, with vacuum vector $v_{0}$ and bilinear form
$\langle\,,\rangle$. Then the $T$-system
  $\tau$-functions are
  \begin{equation}
\tau_{k}=\langle T^{k}v_{0},gv_{0}\rangle,\quad k\in\mathbb{Z},
\label{eq:19}
\end{equation}
where $T$ is (the lift of) a ``translation group element'' in
$GL_{\infty}^{(2)}$, corresponding to the loop group element
$\begin{bmatrix} -z&0\\0&-z\inv
\end{bmatrix}$ in $LGL_{2}$ embedded in $GL_{\infty}^{(2)}$. See
sections \ref{sec:n-comp-ferm}--\ref{sec:tau-funct-outl} for more details.

The $\tau$-functions for the $T$-system turn out to be determinants of finite matrices
with the $c_{k,\ell}$ as entries.

There are discrete shifts, $\tau_{k}^{(\alpha,\beta)}$,
$\alpha,\beta\in\mathbb{Z}$ of the $\tau$-functions, $\tau_{k}$, see
section \ref{sec:tau-funct-outl}. One shows that the shifted $\tau$-functions satisfy
difference equations, which yield the $T$-system of type $A$: For all
$k\ge 0$ and $\alpha,\beta\in \mathbb{Z}$,
\\[.1cm]
\fbox{
    \begin{minipage}{.96\linewidth}
      \begin{equation*}
\tau_{k}^{(\alpha+1,\beta)}\tau_{k}^{(\alpha,\beta+1)}+
\tau_{k+1}^{(\alpha,\beta+1)}\tau_{k-1}^{(\alpha+1,\beta)}=\tau_{k}^{(\alpha,\beta)}\tau_{k}^{(\alpha+1,\beta+1)},
\tag{\text{$2T$-system}}
\end{equation*}
\end{minipage}
}
\\[.2cm]
see Section \ref{sec:2t-relations}. 

Since these equations are related to the group $GL_{\infty}^{(2)}$
acting on 2-component fermionic Fock space we refer to this system of equations as the $2T$-system.

The main point of this paper is that there is an obvious
generalization of the construction of the $2T$-system we just
sketched. Instead of just considering the action of an infinite matrix
group on 2-component fermionic Fock space, one can define a group
$GL_{\infty}^{(n)}$ acting on $n$-component fermionic Fock space
$F^{(n)}$, and consider a lower triangular subgroup of
$GL_{\infty}^{(n)}$. We define $\tau$-functions of type $nT$ as matrix
elements of $n$-component fermionic Fock space. In general these
$\tau$-functions are of the form $\tau_{\mathbf{k}}^{(\bm \beta)}$,
where $\mathbf{k}$ belongs to the root lattice $A_{n-1}$ and
$\bm \beta\in \mathbb{Z}^{n}$. For $n>2$ these $\tau$-functions
satisfy two types of bilinear equations: the ``short'' relations,
described in Section \ref{sec:2+k-term-relations}, have length
$3, 4, \dots n$ (and are absent for $n=2$), and the ``long''
relations, described in Section \ref{sec:long-relations}, of length
$n+1$.

For instance, for $n=3$ the $\tau$-functions are labeled by pairs and
triples of integers $\mathbf{k}=(k,\ell)$ and $\bm\beta=(\alpha,\beta,\gamma)$. An example of a short
relation of type $3T$ is
\[
  \tau_{k,\ell}^{(\alpha,\beta,\gamma)}\tau_{k,\ell-1}^{(\alpha+1,\beta,\gamma)}=
  \tau_{k,\ell-1}^{(\alpha,\beta,\gamma)}\tau_{k,\ell}^{(\alpha+1,\beta,\gamma)}+
  \tau_{k+1,\ell}^{(\alpha,\beta,\gamma)}\tau_{k-1,\ell-1}^{(\alpha+1,\beta,\gamma)},
\]
and an example of a long relation of type $3T$ is 
\begin{equation*}
  \tau_{k,\ell}^{(\alpha+1,\beta,\gamma)}\tau_{k,\ell}^{(\alpha,\beta+1,\gamma+1)}+
  \tau_{k-1,\ell}^{(\alpha+1,\beta,\gamma)}\tau_{k+1,\ell}^{(\alpha,\beta+1,\gamma+1)}+
  \tau_{k-1,\ell-1}^{(\alpha+1,\beta,\gamma)}\tau_{k+1,\ell+1}^{(\alpha,\beta+1,\gamma+1)}=
  \tau_{k,\ell}^{(\alpha,\beta,\gamma)}\tau_{k,\ell}^{(\alpha+1,\beta+1,\gamma+1)}.
\end{equation*}
The $nT$ $\tau$-functions $\tau_{\mathbf{k}}^{(\bm\beta)}(g)$ depend on an
element $g$ in a lower triangular subgroup of the infinite matrix
group $GL_{\infty}^{(n)}$, which contains a subgroup
isomorphic to the loop group $LGL_{n}$. We say that the
$\tau$-functions corresponding to these subgroups, and the equations
satisfied by them, are of type $nQ$. (The relation between $nT$ and
$nQ$ systems is similar to the relation between the KP and $n$-KdV
hierachies, as mentioned above for $n=2$.) In the case that
$\tau_{\mathbf{k}}^{(\bm\beta)}(g)$ is of type $nQ$, i.e., $g$ belongs to
the loop group, it has a symmetry: in that case we have
\[
\tau_{\mathbf{k}}^{(\bm \beta)}=\tau_{\mathbf{k}}^{(\bm \beta+\bm 1)},
\]
where $\bm 1=(1,1,\dots,1)\in\mathbb{Z}^{n}$. For instance, by imposing $\tau_{k}^{(\alpha,\beta)}=\tau_{k}^{(\alpha+1,\beta+1)}$ in the $2T$-system above, we obtain the following $2Q$ relations: 
For all
$k\ge0,\alpha\in \mathbb{Z}$, putting
$\tau_{k}^{(\alpha)}:=\tau_{k}^{(\alpha,0)}$, 
\\[.1cm]
\fbox
{
\begin{minipage}{.96\linewidth}
  \begin{equation*}
    (\tau_{k}^{(\alpha)})^{2}=
    \tau_{k}^{(\alpha+1)}\tau_{k}^{(\alpha-1)}+\tau_{k+1}^{(\alpha-1)}\tau_{k-1}^{(\alpha+1)}.\tag{\text{$2Q$-system}}
  \end{equation*}
\end{minipage}
}
\\[.1cm]
The equations of type $2Q$ and $3Q$ were obtained in earlier work
\cite{taubirkhoff}.
The construction in \cite{taubirkhoff} is a special case of the
construction in this paper. The main technical difference is that in
\cite{taubirkhoff} we used the Birkhoff factorization of a loop group
element, whereas we use the Gauss factorization of an infinite
invertible matrix in this paper.

\section{$n$-Component Fermionic Fock Space.}
\label{sec:n-comp-ferm}
We start by recalling for the reader's convenience
some notation on fermionic Fock space and related constructions, see
our previous paper \cite{taubirkhoff}; we learned this material from \cite{MR1099256}.

Let $e_{0}=
\begin{pmatrix}
  1\\0\\ \vdots
\end{pmatrix}
$, $e_{1}=
\begin{pmatrix}
  0\\1\\\vdots
\end{pmatrix}
$,\dots, $e_{n-1}=
\begin{pmatrix}
  0\\\vdots\\0\\1
\end{pmatrix}$
be the standard basis of $\mathbb{C}^{n}$ and let
$H^{(n)}=\mathbb{C}^{n}\otimes \mathbb{C}[z,z\inv]$, with basis
$e_{a}^{k}=e_{a}z^{k}$, for $a=0,1,\dots,n-1$ and $k\in\mathbb{Z}$.

Let $F^{(n)}$ be the fermionic Fock space, the semi-infinite wedge space based on $H^{(n)}$. It contains semi-infinite wedges
\[
\omega=w_{0}\wedge w_{1}\wedge w_{3}\wedge\dots,\quad w_{i}\in H^{(n)},
\]
where the $w_{i}$ satisfy some restrictions that we will discuss
momentarily. Semi-infinite wedges obey the usual rules of exterior algebras,
like multilinearity in each factor and antisymmetry under exchange of
two factors.

To formulate the restrictions on the vectors $w_{i}$ appearing in the
wedge $\omega$ above, we introduce the Clifford algebra $Cl^{(n)}$
acting on $F^{(n)}$: it is generated by wedging (creation) operators
$e_{a}z^{k}\wedge$, $0\le a\le n-1, k\in\mathbb{Z}$ and their
adjoints, the contracting (annihilation) operators $i(e_{a}z^{k})$
(defined by $i(e_{a}z^{k})\alpha=\beta$ if
$e_{a}z^{k}\wedge\beta=\alpha$).

Let $v_{0}$ be the vacuum vector
\[
v_{0}=
\begin{pmatrix}
  1\\0\\\vdots\\0
\end{pmatrix}\wedge
\begin{pmatrix}
  0\\1\\\vdots\\0
\end{pmatrix}\wedge\dots\wedge
\begin{pmatrix}
  0\\\vdots\\0\\1
\end{pmatrix}\wedge
\begin{pmatrix}
  z\\0\\\vdots\\0
\end{pmatrix}\wedge
\begin{pmatrix}
  0\\z\\\vdots\\0
\end{pmatrix}\wedge\dots.
\]
$F^{(n)}$ is then defined as the span of the wedges obtained by acting on the
vacuum, $v_{0}$, by monomials in the wedging and contracting operators. To
get a basis for $F^{(n)}$ we specify an ordering on the
wedging/contracting operators acting on $F^{(n)}$. One possibility is
the following:

\begin{defn}\label{Def:ElemntaryWedge}
  An \emph{elementary wedge} in $F^{(n)}$ is an element $\omega=Mv_{0}$,
  where
\[
M=M_{n-1}\dots M_{1}M_{0},\quad M_{a}=M_{a}^{+}M_{a}^{-}, \quad 0\le
a\le n-1,
\]
where 
\[
M_{a}^{\pm}=\fermion{a}{k_{1}}{\pm}\fermion{a}{k_{2}}{\pm}\dots
\fermion{a}{k_{s}}{\pm},\quad k_{1}<k_{2}<\dots<k_{s}\le1,
\] 
is a monomial in $\fermion{a}{k}{\pm}$ for $k\le -1$, ordered in
increasing order from left to right. Here,
$\fermion{a}{k}{+}=e_az^k\wedge$ and
$\fermion{a}{k}{-}=i(e_az^{-k-1}).$
\end{defn}

The statement that the elementary wedges form a basis for $F^{(n)}$
follows from the Poincar\'{e}-Birkhoff-Witt theorem for the Lie
superalgebra underlying the Clifford algebra.

We define a bilinear form, denoted $\langle\,,\rangle$, on
$F^{(n)}$ by declaring the elementary wedges to be orthonormal. Then
the fermion operators satisfy the following adjointness property
\begin{equation}\label{eq:8}
  \langle \fermion{a}{k}{+}v,w \rangle=\langle e_{a}z^{k}\wedge
  v,w\rangle=\langle v,i(e_{a}z^{k})w\rangle=\langle v,\fermion{a}{-k-1}{-}w\rangle,
\end{equation} for all $v,w\in F^{(n)}$.

The $n$-component fermionic Fock space $F^{(n)}$ has a grading by the
Abelian group $\mathbb{Z}^{n}$, i.e., we have a decomposition
$F^{(n)}=\mathop{\oplus}\limits_{\delta\in\mathbb{Z}^{n}}F^{(n)}_{\delta}$.  The vacuum
has degree $(0,\dots,0)$. We introduce a basis in $\mathbb{Z}^{n}$ by
\begin{equation}
\delta_{a}=(0,\dots,0,1,0,\dots,0),\quad \text{1 on the $a$th
  position}, \quad 0\le a\le n-1.
\label{eq:28}
\end{equation}
The grading on $F^{(n)}$ induces a grading on linear maps on
$F^{(n)}$: if $L\colon F^{(n)}\to F^{(n)}$ has the property that there
exists a $\delta\in \mathbb{Z}^{n}$ so that for all
$\Omega\in \mathbb{Z}^{n},$ $L$ restricts to a map
$ F^{(n)}_{\Omega}\to F^{(n)}_{\Omega+\delta}$, then we say that $L$ has
degree $\delta$.  Then the wedging operators $e_{a}z^{k}\wedge$ have
degree $\delta_{a}$, and the contracting operators $i(e_{a}z^{k})$
have degree $-\delta_{a}$. The \emph{total degree} of an element $v\in F^{(n)}$ of
degree $(d_{0},d_{1},\dots,d_{n-1})$ is just the sum $\sum_{a=0}^{n-1}d_{a}$.

\section{Fermion Fields.}
\label{sec:fermion-fields}
It is useful to collect the generators of the Clifford algebra in
generating series. Therefore, define \emph{fermion fields}
\[
\psi_{a}^{\pm}(w)=\sum_{k\in\mathbb{Z}}
{}_{a}\psi_{(k)}^{\pm}w^{-k-1},\quad 0\le a\le n-1.
\] 
(This way of labelling components of fermion fields is inspired by the
theory of vertex algebras, see e.g., \cite{MR1651389}. There
are many other conventions, see for instance \cite{tKvdL:BosFer} or
\cite{MR996026}.)

The fermionic fields, $\psi^{\pm}_{a}(z)$ have degree
$\pm\delta_{a}$ and satisfy anti-commutation relations
\[
[\psi_{a}^{\pm}(z),\psi^{\pm}_{b}(w)]_{+}=0,\quad
[\psi_{a}^{+}(z),\psi_{b}^{-}(w)]_{+}=\delta_{ab}\delta(z,w),
\]
where the formal delta distribution is defined by 
\[\delta(z,w)=\sum_{k\in\mathbb{Z}}z^{k}w^{-k-1}.
\]
From \eqref{eq:8} we find adjointness for fields:
\begin{equation}\label{eq:9}
  \langle \psi_{a}^{+}(z)v,w\rangle=\sum_{k\in\mathbb{Z}}\langle
  \fermion{a}{k}{+}v,w \rangle z^{-k-1}=\sum_{k\in\mathbb{Z}}\langle
  v,\fermion{a}{-k-1}{-}w\rangle z^{-k-1}=\langle
  v,\psi_{a}^{-}(z\inv)w\rangle z\inv.
\end{equation}
\section{Fermionic Translation Operators.}
\label{sec:ferm-transl-oper}
 
Besides the action of fermion operators, $\fermion{a}{k}{\pm},$ we also have the action of fermionic translation operators, $Q_{a}\colon F^{(n)}\to
F^{(n)}$, $0\le a\le n-1$  on $F^{(n)}$. These invertible operators are given by
\[
Q_{a}^{\pm1}v_{0}=\psi_{a}^{\pm}(z)v_{0}\mid_{z=0},
\]
and
\begin{align}
\psi_{a}^{\pm}(z)Q_{a}&=z^{\pm1}Q_{a}\psi_{a}^{\pm1}(z), \label{eq:26}\\
\psi_{a}^{\pm}(z)Q_{b}&=-Q_{b}\psi_{a}^{\pm1}(z),\quad a\ne b, \label{eq:24}\\
Q_{a}Q_{b}&=-Q_{b}Q_{a},\quad a\ne b.  \label{eq:79}
\end{align}
The $Q_{a}^{\pm 1}$ have degree $\pm\delta_{a}$.

The $Q_{a}$ are unitary for the standard bilinear form of $F^{(n)}$:
\begin{equation}
\langle Q_{a}v,w\rangle=\langle v,Q_{a}\inv w\rangle, \quad 0\le a\le n-1.
\label{eq:21}
\end{equation}
The group $\langle Q_{a}\rangle$ of automorphisms of $F^{(n)}$
generated by $Q_{a},$ $0\le a\le n-1$ has a basis labelled by
$\bm\beta\in\mathbb{Z}^{n}$: we define
\begin{equation}
  \label{eq:69}
  Q^{\bm\beta}=Q_{0}^{\beta_{0}}Q_{1}^{\beta_{1}}\dots
  Q_{n-1}^{\beta_{n-1}},\quad \bm\beta=\sum_{a=0}^{n-1} \beta_{a}\delta_{a}.
\end{equation}
The multiplication in this basis is simple.
\begin{lem}\label{lem:CentralExtLattice} Let
  $\bm\beta,\bm\beta^{\prime}\in\mathbb{Z}^{n}$. Then
  \begin{equation}
    \label{eq:70}
    Q^{\bm\beta}Q^{\bm\beta^{\prime}}=\epsilon(\bm\beta,\bm\beta^{\prime})Q^{\bm\beta+\bm\beta^{\prime}},
  \end{equation}
  where 
  \begin{equation}
    \label{eq:71}
    \epsilon(\bm\beta,\bm\beta^{\prime})=(-1)^{\sum_{a=0}^{n-2}\beta^{\prime}_{a}\left(\sum_{b=a+1}^{n-1}\beta_{b}\right)}=
(-1)^{\sum_{a=1}^{n-1}\beta_{a}\left(\sum_{b=0}^{a-1}\beta^{\prime}_{b}\right)}.
  \end{equation}
\end{lem}
\begin{proof}
  This follows directly from the definition \eqref{eq:69}, using the
  relation \eqref{eq:79}.
\end{proof}

\begin{lem}\label{lem:propcocycle}
  The function $\epsilon\colon\mathbb{Z}^{n}\times
\mathbb{Z}^{n}\to\{\pm 1\}$ satisfies
\begin{enumerate}
\item $\epsilon(\bm\beta,\bm0)=\epsilon(\bm0,\bm\beta)=1$
\item (Bimultiplicativity)
  \begin{align*}
    \epsilon(\bm\beta_{1}+\bm\beta_{2},\bm\beta_{3})&=\epsilon(\bm\beta_{1},\bm\beta_{3})\epsilon(\bm\beta_{2},\bm\beta_{3}),\\
    \epsilon(\bm\beta_{1},\bm\beta_{2}+\bm\beta_{3})&=\epsilon(\bm\beta_{1},\bm\beta_{2})\epsilon(\bm\beta_{1},\bm\beta_{3}).
  \end{align*}
\item $\epsilon(-\bm\beta_{1},\bm\beta_{2})=\epsilon(\bm\beta_{1},-\bm\beta_{2})=\epsilon(\bm\beta_{1},\bm\beta_{2})$,
\end{enumerate}

\end{lem}
\begin{proof}
  This lemma follows from the explicit expressions \eqref{eq:71} for $\epsilon$.
\end{proof}
It follows from the lemma that the function
$\epsilon\colon\mathbb{Z}^{n}\times \mathbb{Z}^{n}\to\{\pm 1\}$ is a
cocycle on $\mathbb{Z}^{n}$, well known in representation theory, cf.,
\cite{FrKa:BasRepDu}, and the theory of lattice vertex algebras, cf.,
\cite{MR996026}, \cite{MR1651389}.  The group $\langle Q_{a}\rangle$
generated by the fermionic translation operators is a central
extension of the lattice $\mathbb{Z}^{n}$.

The function $B\colon \mathbb{Z}^{n}\times
\mathbb{Z}^{n}\to\{\pm 1\}$ given by
\[
B(\bm\beta,\bm\beta^{\prime})=\epsilon(\bm\beta,\bm\beta^{\prime})\epsilon(\bm\beta^{\prime},\bm\beta)
\]
controls the commuting of basis elements in $\langle Q_{a}\rangle$: we have
\begin{equation}
  \label{eq:77}
Q^{\bm\beta}Q^{\bm\beta^{\prime}}=B(\bm\beta,\bm\beta^{\prime})Q^{\bm\beta^{\prime}}Q^{\bm\beta}.
\end{equation}
It is well know that\footnote{The Abelian group $\mathbb{Z}^{n}$ has a bilinear form given by
\[
\delta_{a}\cdot\delta_{b}=\delta_{ab},\quad 0\le a,b\le n-1.
\]}
\begin{equation}
  \label{eq:78}
  B(\bm\beta,\bm\beta^{\prime})=(-1)^{\bm\beta\cdot\bm\beta^{\prime}+(\bm\beta\cdot\bm\beta)(\bm\beta^{\prime}\cdot\bm\beta^{\prime})}.
\end{equation}
(To show this observe that $\epsilon$ and hence $B$ are bimultiplicative, as
is the right hand side of \eqref{eq:78}. Then check that \eqref{eq:78}
is true for the basis $\{\delta_{a}\}_{a=0}^{n-1}$ of
$\mathbb{Z}^{n}$. Cf., \cite{MR1651389}.)

In the proof of the long relations for
$nT$ $\tau$-functions, see Theorem \ref{thm:long-relations-1},  we will have to consider tensor products of the form
$Q^{\bm\beta}\otimes Q^{\bm\beta^{\prime}}$ (acting on $F^{(n)}\otimes
F^{(n)}$).

\begin{lem}\label{lem:TensorQ}
  Let $\bm\beta\in\mathbb{Z}^{n}$, $a,b\in\{0,1,\dots,n-1\}$,
  $n_{a},n_{b}\in\mathbb{Z}$. Then
\[
Q^{\bm\beta}Q_{a}^{n_{a}}\otimes
Q^{\bm\beta}Q_{b}^{n_{b}}=\epsilon(\bm\beta,n_{a}\delta_{a}+n_{b}\delta_{b})Q^{\bm\beta+n_{a}\delta_{a}}\otimes Q^{\bm\beta+n_{b}\delta_{b}}.
\]
\end{lem}
\begin{cor}\label{cor:tensorQQ}
\[
Q^{\bm\beta}Q_{a}^{\pm1}\otimes Q^{\bm\beta}Q_{a}^{\mp1}=Q^{\bm\beta\pm\delta_{a}}\otimes Q^{\bm\beta\mp\delta_{a}}.
\]
  
\end{cor}

\section{Root Lattice and Translation Operators.}
\label{sec:root-latt-transl}

The lattice $\mathbb{Z}^{n}$ contains the root lattice $A_{n-1}$:
\[
  A_{n-1}=\sum\limits_{i=1}^{n-1}\mathbb{Z}\alpha_{i},\quad
  \alpha_{i}=\delta_{i-1}-\delta_{i}.
\]
 We will call nonzero elements in $A_{n-1}$ of the form
$\alpha=\sum n_{i}\alpha_{i}$ \emph{positive} roots if
all $n_{i}\ge0$. 
 
There is subgroup of $\langle Q_{a}\rangle$ corresponding to
$A_{n-1}\subset \mathbb{Z}^{n}$, generated
by
\begin{equation}
T_{\alpha_{i}}=Q_{i-1}\inv Q_{i}=Q^{-\alpha_{i}}.
\label{eq:72}
\end{equation}
The degree of $T_{\alpha_{i}}$ is (regrettably)
$\delta_{i}-\delta_{i-1}=-\alpha_{i}$. We will often
write just $T_{i}$ for $T_{\alpha_{i}}$.

The translation operator $T_i$ is unitary, just as the fermionic
translation operators are: from \eqref{eq:21} it follows that
\begin{equation}
  \label{eq:22}
  \langle T_iv, w\rangle=\langle v,T_i\inv w\rangle.
\end{equation}

Just as for $\langle Q_{a}\rangle$ we choose a basis for
$\langle T_{\alpha_{i}}\rangle$: if
$\mathbf{k}=\sum_{i=1}^{n-1}k_{i}\alpha_{i}\in A_{n-1}$ then we define
\begin{equation}
  \label{eq:73}
  T^{\mathbf{k}}=Q^{-\mathbf{k}}.
\end{equation}
To write
\[
T^{\mathbf{k}}=Q_{0}^{K_{0}}Q_{1}^{K_{1}}\dots Q_{n-1}^{K_{n-1}},
\]
we put
\[
-\mathbf{k}=\sum_{a=0}^{n-1}K_{a}\delta_{a},
\]
so that
\begin{equation}
  \label{eq:89}
  K_{a}=-\mathbf{k}\cdot \delta_{a}=k_{a}-k_{a+1},
\end{equation} where $k_0=k_{n}=0$.

\begin{lem}\label{Lem:GeneratorsTranslationGroupIdentities}
  \  \\
\begin{enumerate}
  \item \label{item:7} $Q_{a}^{\pm1}(T^{\mathbf{k}})^{\pm1}=(-1)^{\delta_{a}\cdot\mathbf{k}}(T^{\mathbf{k}})^{\pm1}Q_{a}^{\pm1}$.
\item \label{item:9} For $\mathbf{k},\mathbf{k}^{\prime}\in A_{n-1}$ we have
\[
T^{\mathbf{k}}T^{\mathbf{k}^{\prime}}=
\epsilon(\mathbf{k},\mathbf{k}^{\prime})
T^{\mathbf{k}+\mathbf{k}^{\prime}}.
\]
\end{enumerate}
\end{lem}
\begin{proof}
   Part (\ref{item:7}) is a special case of \eqref{eq:78}, noting that
   $\mathbf{k}\cdot\mathbf{k}\in2\mathbb{Z}$ for $\mathbf{k}\in A_{n-1}$.  
   For Part (\ref{item:9}) we use definition \eqref{eq:73} and Lemma
   \ref{lem:CentralExtLattice}:
\[
     T^{\mathbf{k}}T^{\mathbf{k}^{\prime}}
      =Q^{-\mathbf{k}}Q^{-\mathbf{k}^{\prime}}=\epsilon(-\mathbf{k},-\mathbf{k}^{\prime})Q^{-(\mathbf{k}+\mathbf{k}^{\prime})}=\epsilon(\mathbf{k},\mathbf{k}^{\prime})T^{\mathbf{k}+\mathbf{k}^{\prime}}.
    \]
 \end{proof}

\section{The Lie Algebra $\glninf$ and its Completion and Central Extension.}
\label{sec:lie-algebra-gl_infty}

Define the Lie algebra $\glninf$ as the Lie algebra of linear maps on $H^{(n)}$
generated by $E_{ab}^{k,\ell}$, $0\le a,b\le n-1, k,\ell\in\mathbb{Z}$, where
\[
E_{ab}^{k,\ell}e_{c}z^{m}=\delta_{bc}\delta_{\ell m}e_{a}z^{k}.
\]
Then $\glninf$ also acts on $F^{(n)}$ by
\begin{equation}
E_{ab}^{k,\ell}\mapsto (e_{a}z^{k}\wedge)(i(e_{b}z^{\ell}))={}_{a}\psi^{+}_{(k)}{}_{b}\psi^{-}_{(-\ell-1)}.
\label{eq:40}
\end{equation}
Introduce generating series for Lie algebra elements acting on $F^{(n)}$ by
\begin{equation}
E_{ab}(z,w)=\sum_{k,\ell\in\mathbb{Z}}E_{ab}^{k,\ell}z^{-k-1}w^{\ell}.
\label{eq:18}
\end{equation}
We have an expression in terms of fermion fields for the generating
series:
\begin{equation}
\label{eq:90}
E_{ab}(z,w)=\sum_{k,\ell \in\mathbb{Z}}{}_{a}\psi^{+}_{(k)}{}_{b}\psi^{-}_{(-\ell-1)}z^{-k-1}w^{\ell}
=\psi_{a}^{+}(z)\psi_{b}^{-}(w).
\end{equation} 
The series $E_{ab}(z,w)$ has degree $\delta_{a}-\delta_{b}$.

If we want to discuss, for instance, the loop algebra of $gl_{n}$ in
this context, then  we need to allow infinite sums of the
$E_{ab}^{k,\ell}$. (See Kac-Peterson, \cite{KaPe:LectinfwedgMKP}, for a
detailed discussion of various completions of $gl_{\infty}^{(n)}$ in
the case $n=1$.)  So define $\overline{a}_{\infty}^{\,(n)}$ to be the Lie
algebra of sums $X=\sum_{k,\ell}X_{ab}^{k,\ell}E_{ab}^{k,\ell}$, such that for
each $k\in\mathbb{Z}$ there are only finitely many $X_{ab}^{i,j}\ne 0$
with $i\le k$ and $j\ge k$.  These matrices no longer act on
$H^{(n)}$, so we consider the completion
$\overline{H}^{(n)}=\mathbb{C}^{n}\otimes
\mathbb{C}((z))=\mathbb{C}^{n}\otimes \mathbb{C}[[z]][z\inv]$. The
spaces $H^{(n)}$ and $\overline{H}^{(n)}$ have the same semi-infinite
wedge space $F^{(n)}$.  The matrices in $\bar{a}_{\infty}^{(n)}$ do act on
$\overline{H}^{n}$ but not on $F^{(n)}$, in general. We therefore
change the action of $E_{ab}^{k,\ell}$ on $F^{(n)}$ by putting
\begin{equation}
E_{ab}(z,w)=\colon \psi_{a}^{+}(z)\psi_{b}^{-}(w)\colon,
\label{eq:87}
\end{equation}
where the normal ordering $\colon\,\colon$ moves annihilation
operators to the right of creation operators. This gives rise to a
central extension of $\overline{a}_{\infty}^{(n)}$:
\begin{equation}
0\to \mathbb{C}\to a_{\infty}^{(n)}\overset{\pi}\to\overline{a}_{\infty}^{(n)}\to 0.
\label{eq:17}
\end{equation}
Then $a_{\infty}^{(n)}$ acts on $F^{(n)}$.

\section{The Group $GL^{(n)}_\infty$ and its Completion and Central Extension.}
\label{sec:group-gl_infty-compl}

The group analog of the Lie algebra
  $gl_{\infty}^{(n)}$ is $GL^{(n)}_{\infty}$, the group of invertible infinite matrices
that differ from the identity by an element of
$gl^{(n)}_{\infty}$. $GL^{(n)}_{\infty}$ acts on
$H^{(n)}=\mathbb{C}^{n}\otimes \mathbb{C}[z,z\inv]$. Again we want to
consider some completion where we allow infinite sums. 
Define the group $\overline{A}_{\infty}^{\,(n)}$ as the set of
invertible elements in $\overline{a}_{\infty}^{\,(n)}$, thought of as an
associative algebra. Then $\overline{A}_{\infty}^{\,(n)}$ acts on
$\overline{H}^{\,(n)}$, but not on $F^{(n)}$. There is a central
extension $A_{\infty}^{(n)}$ of $\overline{A}^{\,(n)}_{\infty}$,
\[
1\to \mathbb{C}^{\times}\to A^{(n)}_{\infty}\overset{\pi}\to
\overline{A}_{\infty}^{\,(n)}\to 1,
\]
so that $A_{\infty}^{(n)}$ does act
on $F^{(n)}$, see \cite{KaPe:LectinfwedgMKP}.

\section{Loop Algebra Embedding.}
\label{sec:loop-algebra-embedd}
In this section we describe an embedding of the (Laurent polynomial)
loop algebra $Lgl_{n}=gl_{n}\otimes \mathbb{C}[z,z\inv]$ in
$\overline{a}_{\infty}^{\,(n)}$. There are two reasons for this. The
first is that this embedding allows us to give an explicit description of
the (projections of) the fermionic translation elements $Q_{a}$ and
the translation group elements $T_{i}$. The other reason is that the
hierarchies of difference equations that we attach to
$GL_{\infty}^{(n)}$ have a reduction attached to the loop group
$LGL_{n}$, see section \ref{sec:intro}.

Consider the Laurent polynomial loop algebra $Lgl_{n}$. It is
  generated by elements $z^{k}E_{ab}$ which act in the obvious way on
  $H^{(n)}$, with
  \begin{equation}
z^{k}E_{ab}\mapsto \sum_{\ell\in\mathbb{Z}}E^{k+\ell,\ell}_{ab}.
\label{eq:29}
\end{equation}
This gives an embedding
$Lgl_{n}\hookrightarrow \overline{a}_{\infty}^{\,(n)}$, which gives
rise, using \eqref{eq:17}, to the usual central extension of the loop
algebra:
\[
0\to \mathbb{C}\to \widehat{Lgl}_{n}\overset{\pi}\to Lgl_{n}\to 0.
\]
Now we are going to relate the fermionic translation
  operators $Q_{a}$ to loop group elements. Recall the fermion creation
  operator $\fermion{a}{k}{+}=e_{a}z^{k}\wedge$. Define a map
  $\psi^{+}\colon H^{(n)}\to \operatorname{End}(F^{(n)})$, $v\mapsto
  v\wedge$, so that in particular
  $\fermion{a}{k}{+}=\psi^{+}(e_{a}z^{k})$. We can extend the domain of $\psi^{+}$
  to $\overline{H}^{\,(n)}$. Then a fundamental property of the action
$R\colon A_{\infty}^{(n)}\to \operatorname{Aut}(F^{(n)})$ is that for
all $v\in \overline{H}^{\,(n)}, g\in A_{\infty}^{(n)}$ we have
\begin{equation}
R(g)\psi^{+}(v)R(g\inv)=\psi^{+}(\pi(g)\cdot v).
\label{eq:15}
\end{equation}
Here $\gamma\cdot v$ is the action of
$\gamma\in\overline{A}_{\infty}^{\,(n)}$ on $v\in
\overline{H}^{\,(n)}$, and $\pi\colon A^{(n)}_{\infty}\to
\overline{A}_{\infty}^{\,(n)}$ is the projection.  See  Kac-Peterson, \cite{KaPe:LectinfwedgMKP},
for details.

This allows us to calculate the projections of
the fermionic translation operators. We find, using \eqref{eq:26} and
\eqref{eq:15}, for all $a=0,1,\dots,n-1,$
\begin{equation}
\label{eq:20}
  \begin{aligned}
\overline Q_{a}&=\pi(Q_{a})=z\inv E_{aa}-\sum_{b\ne a}E_{bb},\\
&=\sum_{k\in\mathbb{Z}}\left( E_{aa}^{k-1,k}-\sum_{b\ne a}E_{bb}^{k,k}\right).
\end{aligned}
\end{equation}
Similarly, for all $i=1,\dots, n-1$,
\begin{equation}
  \begin{aligned}
    \overline T_{i}&=\pi(T_{i})=-zE_{i-1,i-1}-z\inv E_{ii}+\sum_{b\ne
      i-1,i}E_{bb},
    \\
    &=\sum_{k\in\mathbb{Z}}\left(
      -E_{i-1,i-1}^{k+1,k}-E_{i,i}^{k-1,k}+\sum_{b\ne
        i,i-1}E_{bb}^{k,k}\right).
  \end{aligned}
\label{eq:27}
\end{equation}
\section{Gauss Factorization and Fermion Matrix Elements.}
\label{sec:gauss-fact-ferm}

The decomposition 
$\overline{H}^{\,(n)}=\overline{H}^{\,(n)}_{+}\oplus
\overline{H}^{\,(n)}_{-}$, where
$\overline{H}^{\,(n)}_{+}=\mathbb{C}^{n}\otimes\mathbb{C}[[z]]$ and
$\overline{H}^{\,(n)}_{-}=\mathbb{C}^{n}\otimes\mathbb{C}[z\inv]z\inv$,
induces a block decomposition on elements
$a\in \overline{a}^{\,(n)}_{\infty}$: every such $a$ can be written as
\begin{equation}
a=
\begin{pmatrix}
  a_{++}&a_{+-}\\a_{-+}&a_{--}
\end{pmatrix},\label{eq:13}
\end{equation}
where $a_{ij}\colon \overline H^{\,(n)}_{j}\to \overline H^{\,(n)}_{i}$, $i,j\in\{-,+\}$.

Recall the group $\overline{A}^{\,(n)}_{\infty}$ of invertible
elements in the associative algebra $\overline{a}^{\,(n)}_{\infty}$;
$\overline{A}_{\infty}^{\,(n)}$ acts on $\overline{H}^{\,(n)}$ and the
central extension $A_{\infty}^{(n)}$ acts on $F^{(n)}$,
see section \ref{sec:group-gl_infty-compl}.

We say that $g\in \overline{A}^{\,(n)}_{\infty}$ has a \emph{Gauss Decomposition} if we
can factor $g$ as follows:
\begin{equation}
g=g_{-}g_{0+},\quad g_{-}=
\begin{pmatrix}
  1&0\\
X_{-+}&1
\end{pmatrix},\quad g_{0+}=
\begin{pmatrix}
  g_{++}&g_{+-}\\0&E_{--}
\end{pmatrix},
\label{eq:39}
\end{equation}
where $g_{-},g_{0+}$ are invertible, with inverses of the same
form. 
Let $\hat g\in A_{\infty}^{(n)}$ be any lift of $g\in\overline{A}^{\,(n)}_{\infty}$, i.e.,
$\pi(\hat g)=g$. Define $\tau=\langle v_{0},\hat gv_{0}\rangle$. Now $g$
has a Gauss factorization if and only if $\tau\ne0$. In this case ($\tau\ne 0$) we can give an explicit formula for the
factor $g_{-}$ of the Gauss factorization of $g$, in terms of matrix
elements with fermion fields inserted.

\begin{lem}
\label{lem1}
  Let $g\in \overline{A}^{\,(n)}_{\infty}$ have a Gauss factorization (so that
  $\tau\ne0$) with negative component $g_{-}=
  \begin{pmatrix}
    1&0\\X&1
  \end{pmatrix}$. Define 
\[
g_{ab}(z,w)=\langle v_{0},
\colon 
\psi_{b}^{+}(w)\psi_{a}^{-}(z)
\colon 
gv_{0}\rangle/\tau.
\]
Then 
\[
X=\sum_{0\le a,b\le n-1}\Res_{z,w}\left(g_{ab}(z,w)E_{ab}(z,w)\right),
\] 
where $E_{ab}(z,w)$ is the generating series \eqref{eq:18} of Lie algebra elements.
\end{lem}

\begin{proof}
  Write $X=\displaystyle\sum_{0\le a,b\le n-1}X_{ab}$, where
\[
X_{ab}=\sum_{r,s\ge0}x^{r,s}_{ab}E_{ab}^{-r-1,s}.
\]
Introduce generating series
\[
x_{ab}(z,w)=\sum_{r,s}x_{ab}^{r,s}z^{-r-1}w^{-s-1}.
\]
so that
\[
X_{ab}=\Res_{z,w}(x_{ab}(z,w)E_{ab}(z,w)).
\]
Here 
\[
\Res_{z,w}=\Res_{z}\Res_{w},
\]
and $\Res_{z}$ is the coefficient of $z\inv$ in a series in $z,z\inv$.

Now we can calculate the coefficients $x^{r,s}_{ab}$ using the
fermionic Fock space $F^{(n)}$. Here we prove the formula for $x_{ab}(z,w)$ in the case that $a\ne b$. Since formulas for the case that $a=b$ are not needed in this paper, we will leave the details to the reader. Consider $g_{-}v_{0}$. First note
that
\begin{align*}
  g_{-}v_{0}=v_{0}+\sum_{a,b}\sum_{r,s\ge0}x_{ab}^{r,s}E_{ab}^{-r-1,s}v_{0}+\dots,
\end{align*}
where the omitted terms are quadratic and higher in the $E_{ab}^{i,j}$s.  Next note
that if $g$ has Gauss factorization, then
$g_{-}v_{0}=\hat gv_{0}/\tau$. Hence
\[
x^{r,s}_{ab}=\langle E_{ab}^{-r-1,s}v_{0},g_{-}v_{0}\rangle=\langle E_{ab}^{-r-1,s}v_{0},gv_{0}\rangle/\tau.
\]
We then get the generating series by \eqref{eq:40} and \eqref{eq:90}
\begin{align*}
x_{ab}(z,w)&=\sum_{r,s\ge0}z^{-r-1}w^{-s-1}\langle
{}_{a}\psi^{+}_{(-r-1)}{}_{b}\psi^{-}_{(-s-1)}v_{0},gv_{0}\rangle/\tau=\\
&=z\inv w\inv\langle 
\psi_{a}^{+}(z\inv)\psi_{b}^{-}(w\inv)
v_{0},gv_{0}\rangle/\tau.
\end{align*}
Then using the adjointness property \eqref{eq:9} we move the normal
ordered product to the other side to find
\[
x_{ab}(z,w)=\langle v_{0},
\psi_{b}^{+}(w)\psi_{a}^{-}(z)
gv_{0}\rangle/\tau,
\]
i.e., $x_{ab}(z,w)=g_{ab}(z,w)$ as we wanted to show.
\end{proof}

\section{A Lower Triangular Subgroup, Baker Functions and Connection Matrices.}
\label{sec:nilpotent-subgroup}

Consider the infinite matrix
\begin{equation}
g=1_{H^{(n)}}+\sum_{a>b}\sum_{k,\ell \in
  \mathbb{Z}}c_{k,\ell}^{a,b}E_{a,b}^{-k-1,\ell}.
\label{eq:11}
\end{equation}
Here we can think of the $c_{k,\ell}^{a,b}$ as complex numbers. If we
impose the condition that all but finitely many of the
$c_{k,\ell}^{a,b}$ are zero, then $g$ is an element of
$GL^{(n)}_{\infty}$. Similarly, by assuming that for each $k$ there
are only finitely many $c_{i,j}^{a,b}\ne 0$ if $i\le k$ and $j\ge
k$, we obtain an element of $\overline A^{\,(n)}_{\infty}$. In either
case there is a unique lift $\hat g\in A^{(n)}_{\infty}$ of $g$, such
that $\pi(\hat g)=g$ and $\langle v_{0},\hat gv_{0}\rangle =1$. These elements
$\hat g$ form a subgroup $\hat {\mathcal{N}}\subset A^{(n)}_{\infty}$
isomorphic to the group $\mathcal{N}$ of elements $g$: the central
extension is trivial over $\mathcal{N}$. Later we will often
silently identify the two groups $\mathcal{N}$ (acting on $\overline
H^{(n)}$) and $\hat{\mathcal{N}}$ (acting on $F^{(n)}$), and we will
not distinguish between $g$ and $\hat g$.

It will also be useful to think of the $c_{k,\ell}^{a,b}$
appearing in $g$ as formal variables, coordinates on the group
$\mathcal{N}$. In this interpretation $g$ is not, for instance, a
linear map on fermionic Fock space, but a map
$F^{(n)}\to F^{(n)}[[c_{k,\ell}^{a,b}]]$. In particular, matrix elements
involving $g$ will then be formal series in the $c_{k,\ell}^{a,b}$. In
practice this causes no problems. In particular, the Gauss
factorization and Lemma \ref{lem1} still work if we use these
``formal'' group elements.

We are going to define shifts and translates
$g^{(\mathbf{k},\bm \beta)}$ of this $g$ in order to define
\emph{Baker functions} and \emph{Connection Matrices} relating the
Baker functions. In the next section we will establish a criterion for
connection matrices to be nonnegative.  Later, we will express the
connection matrices in $\tau$-functions, and use this to derive our $nT$-equations,
bilinear equations for $\tau$-functions, see sections
\ref{sec:2+k-term-relations} and \ref{sec:long-relations}.

Recall the fermionic translation operators $Q_a$ and translation operators $ T_{i}$ acting on fermionic
Fock space $F^{(n)}$, see sections \ref{sec:ferm-transl-oper} and \ref{sec:root-latt-transl}. Denote
by $\overline Q_{a}=\pi(Q_{a})$ and $\overline T_{i}=\pi(T_{i})$ the projection on
the non-centrally extended completed infinite matrix group $\overline
A_\infty^{\,(n)}$, see \eqref{eq:20}, \eqref{eq:27}. Similar to
\eqref{eq:69} and \eqref{eq:73} we define products: if
\begin{align*}
\bm{\beta}&=\sum_{b=0}^{n-1}\beta_{b}\delta_{b}=(\beta_{0},\dots,\beta_{n-1})\in \mathbb{Z}^{n},\quad \mathbf{k}=\sum k_{i}\alpha_{i}\in A_{n-1},
\end{align*}
then
\begin{equation}\label{eq:75}
\begin{aligned}
\overline T^{\,\mathbf{k}}&=\overline T_{1}^{\, k_{1}}\dots  \overline T_{n-1}^{\, k_{n-1}},&
\overline Q^{\,\bm{\beta}}&=\overline Q_{0}^{\, \beta_{0}}\overline  Q_{1}^{\,\beta_{1}}\dots \overline Q_{n-1}^{\,\beta_{n-1}}.
\end{aligned}
\end{equation}
Note that in contrast to $Q_{a}$ and $T_{i}$ the projection $\overline
Q_{a}$ and $\overline T_{i}$ all commute among themselves, so we do not
need to keep track of their ordering. Also note that $(\overline
Q^{\,\bm \beta})\inv=\overline Q^{\,-\bm\beta}$ and $(\overline
T^{\,\mathbf{k}})\inv=\overline T^{\,-\mathbf{k}}$.

We define
\begin{equation}
\begin{aligned}
  g^{(\bm{\beta})}&=
\overline Q^{\,\bm{\beta}}g \,\overline Q^{\,-\bm{\beta}},&
g^{(\mathbf{k},\bm{\beta})}&=\overline T^{-\mathbf{k}}g^{(\bm{\beta})}.
\end{aligned}\label{eq:30}
\end{equation}
We have Gauss factorization 
\[
g^{(\mathbf{k},\bm{\beta})}=g^{(\mathbf{k},\bm{\beta})}_{-}g^{(\mathbf{k},\bm{\beta})}_{0+}.
\]
We use this to define \emph{Baker Functions}
\begin{equation}
\Psi^{(\mathbf{k},\bm{\beta})}=\overline T^{\,\mathbf{k}}\overline Q^{\,-\bm{\beta}}g^{(\mathbf{k},\bm{\beta})}_{-}.
\label{eq:46}
\end{equation}
Next we construct \emph{connection matrices}. The Baker functions are
all invertible, being elements of $\overline A^{\,(n)}_{\infty}$, so
there are elements
$\Gamma^{(\mathbf{k}^{\prime},\bm{\beta}^{\prime})}_{(\mathbf{k},\bm{\beta})}\in
\overline A^{\,(n)}_{\infty}$
such that
\[
  \Psi^{(\mathbf{k}^{\prime},\bm{\beta}^{\prime})}=\Psi^{(\mathbf{k},\bm{\beta})}
\Gamma^{(\mathbf{k}^{\prime},\bm{\beta}^{\prime})}_{(\mathbf{k},\bm{\beta})}.
\]
Explicitly we have
\begin{equation}
\Gamma^{(\mathbf{k}^{\prime},\bm{\beta}^{\prime})}_{(\mathbf{k},\bm{\beta})}=
\left(g_{-}^{(\mathbf{k},\bm{\beta})}\right)\inv \overline
T^{\,\mathbf{k}^{\prime}-\mathbf{k}}\overline
Q^{\,\bm{\beta}-\bm{\beta}^{\prime}}\left( g_{-}^{(\mathbf{k}^{\prime},\bm{\beta}^{\prime})}\right).
\label{eq:59}
\end{equation}
Now we have 
\[
\overline T^{\,\mathbf{k}^{\prime}-\mathbf{k}}\overline
Q^{\,\bm{\beta}-\bm{\beta}^{\prime}}g^{(\mathbf{k}^{\prime},\bm{\beta}^{\prime})}=g^{(\mathbf{k},\bm{\beta})}\overline Q^{\,\bm{\beta}-\bm{\beta}^{\prime}},
\]
so that
\[
\overline T^{\,\mathbf{k}^{\prime}-\mathbf{k}}\overline
Q^{\,\bm{\beta}-\bm{\beta}^{\prime}}g^{(\mathbf{k}^{\prime},\bm{\beta}^{\prime})}_{-}g^{(\mathbf{k}^{\prime},\bm{\beta}^{\prime})}_{0+}=
g^{(\mathbf{k},\bm{\beta})}_{-}g^{(\mathbf{k},\bm{\beta})}_{0+}\overline Q^{\,\bm{\beta}-\bm{\beta}^{\prime}},
\]
and by rearranging factors we find an alternative expression for the
connection matrices
\begin{equation}
  \label{eq:12}
  \Gamma^{(\mathbf{k}^{\prime},\bm{\beta}^{\prime})}_{(\mathbf{k},\bm{\beta})}=\left(g^{(\mathbf{k},\bm{\beta})}_{0+}\right)
\left(\overline  Q^{\,\bm{\beta}-\bm{\beta}^{\prime}}\right)
\left(g^{(\mathbf{k}^{\prime},\bm{\beta}^{\prime})}_{0+}\right)\inv.
\end{equation}

\section{Nonnegativity.}
\label{sec:positivity}

Recall the block decomposition \eqref{eq:13} of an element
$a\in\overline a^{\,(n)}_{\infty}$. We say that $a$ is \emph{nonnegative} if
the block $a_{-+}$ is zero. In other words, if we expand
$a=\sum_{b,c}\sum_{i,j\in\mathbb{Z}}a^{i,j}_{bc}E_{bc}^{i,j}$, then
$a$ is nonnegative if all components $a_{bc}^{-i-1,j}$ with $i,j\ge
0$ are zero. 

In particular, if $g_{0+}$ is a factor in a Gauss factorization, see
\eqref{eq:39}, then $g_{0+}$ is nonnegative. Also, if $x$ belongs to
the loop algebra $Lgl_{n}$ embedded in $\overline a^{\,(n)}_{\infty}$,
then $x$ is nonnegative if it contains only nonnegative powers of
$z$, see \eqref{eq:29}. In particular, a product
$\overline Q^{\,\bm{\beta}}$ is nonnegative if all components
$\beta_{a}$ of $\bm{\beta}=(\beta_{0},\beta_{1},\dots,\beta_{n-1})$
satisfy $\beta_{a}\le0$, see \eqref{eq:20}. Finally, the product of nonnegative elements is again nonnegative. 

This gives us an easy way to check which of the connection matrices
$\Gamma^{(\mathbf{k^{\prime}},\bm{\beta^{\prime}})}_{(\mathbf{k},\bm{\beta})}$,
see \eqref{eq:12}, are nonnegative: these are the connection matrices
such that $\bm{\beta}-\bm{\beta^{\prime}}\le0$, i.e.,
$\Gamma^{(\mathbf{k^{\prime}},\bm{\beta^{\prime}})}_{(\mathbf{k},\bm{\beta})}$
is nonnegative when
\[
(\beta_{a}-\beta_{a}^{\prime})\le0,\quad 0\le a\le n-1.
\]

\section{$\tau$-Functions.}
\label{sec:tau-funct-outl}

\subsection{Introduction and Outline.}

At this point we have all of the ingredients we need to define the
$\tau$-functions of type $nT$.

Recall the  translation group elements $
\overline T^{\,\mathbf{k}}$ and the shifted group element $g^{(\bm{\beta})}$ of
\eqref{eq:30}, used to define the connection matrices. They lift to
operators on $F^{(n)}$, denoted by $T^{\mathbf{k}}$ and
$g^{(\bm\beta)}$, respectively. We define (in the central extension)
\begin{equation}
  \label{eq:80}
  g^{(\mathbf{k},\bm\beta)}=\left(T^{\mathbf{k}}\right)\inv g^{(\bm\beta)},\quad g^{(\bm\beta)}= Q^{\bm\beta}g(Q^{\bm\beta})\inv.
\end{equation}
Then the $\tau$-function of type $nT$ is
\begin{equation}
  \label{eq:36}
  \tau^{(\bm{\beta})}_{\mathbf{k}}(g)=
\langle v_{0}, g^{(\mathbf{k},\bm\beta)}v_{0}\rangle=
\langle T^{\mathbf{k}}v_{0},g^{(\bm\beta)}v_{0}\rangle.
\end{equation}
We will often write $\tau_{\mathbf{k}}^{(\bm\beta)}$ for $\tau_{ \mathbf{k}}^{(\bm\beta)}(g)$.

We will give residue formulas for the $\tau$-functions, see Theorem
\ref{thm:tauk}. In section \ref{sec:shift-fields-formula}, we will also give explicit formulas in
terms of $\tau$-functions for some of the components of the factors,
$g_{-}^{(\mathbf{k},\bm\beta)},g_{0+}^{(\mathbf{k},\bm\beta)},$ of the
connection matrices appearing in \eqref{eq:59} and \eqref{eq:12}. This
will allow us in sections \ref{sec:2+k-term-relations}
and \ref{sec:long-relations} to derive bilinear equations of  lengths $3, 4,
\dots, n+1$ for $\tau$-functions of type $nT$.

\subsection{Degree and Nonvanishing of $\tau$-functions.}
\label{sec:degree-non-vanishing}

In this subsection we will show that many of the
$\tau$-functions defined by \eqref{eq:36} are identically zero, by
investigating  the degree of the elements $T^{\mathbf{k}}v_{0}$ and
$g^{\bm\beta}v_{0}$. 

First we rewrite the element $g$ of \eqref{eq:11} as a product of
exponentials:
\begin{equation}
g=\overrightarrow{\prod_{a>b}}\exp\left(\Gamma_{ab}\right),
\label{eq:37}
\end{equation}
where
\begin{equation}
\Gamma_{ab}=\Res_{z,w}\left( C_{ab}(z,w)E_{ab}(z,w)\right),
\label{eq:82}
\end{equation}
with
\begin{equation}
C_{ab}(z,w)=\sum_{k,\ell\in\mathbb{Z}}c_{k,\ell}^{a,b}z^{-k-1}w^{-\ell-1},
\label{eq:83}
\end{equation}
and $E_{ab}(z,w)$ is the generating series \eqref{eq:18}. Of course,
the $\Gamma_{ab}$ in \eqref{eq:37} do not commute (because the
$E_{ab}(z,w)$ do not), so we need to prescribe an ordering of the
exponential factors. For instance, we can define the ordered product
by
\begin{equation}
  \overrightarrow{\prod_{a>b}}\exp(\Gamma_{ab})=\exp(\Gamma_{a_{1}b_{1}})\exp(\Gamma_{a_{2}b_{2}})\dots
  \exp(\Gamma_{a_{N}b_{N}}),
\label{eq:44}
\end{equation}
where $N=n(n-1)/2$, $(a_{1},b_{1})=(1,0)$ and 
\[
(a_{i+1},b_{i+1})=
\begin{cases}
  (a_{i},b_{i}+1)& \text{ if $b_{i}<a_{i}-1$}\\
  (a_{i}+1,0)&\text{ if $b_{i}=a_{i}-1$}
\end{cases}
\]
From here on, we will use this convention of ordered product over
$a>b$.

Now note that the degree of $\Gamma_{ab}$ is
$\omega_{ab}=\delta_{a}-\delta_{b}$. This can be written as a
telescoping sum: in the case $a>b$ we have
\begin{equation}
  \label{eq:38}
  \omega_{ab}=(\delta_{a}-\delta_{a-1})+(\delta_{a-1}-\delta_{a-2})+\dots+
(\delta_{b+1}-\delta_{b})=-\sum_{k=b+1}^{a}\alpha_{k}
\end{equation}
Hence 
$\omega_{ab}$ is a negative root (for $A_{n-1}$, see section
\ref{sec:root-latt-transl}) for all $a>b$. This means that $gv_{0}$ is an
infinite sum of terms with degree a negative root (or of degree 0, for
$v_{0}$).

Next consider the shifted group element $g^{(\bm\beta)}=Q^{\bm\beta} g
Q^{-\bm\beta}$. Again this can be written as a product 
\begin{equation}
g^{(\bm\beta)}=\overrightarrow{\prod_{a>b}}\exp\left(\Gamma_{ab}^{(\bm\beta)}\right),\quad 
\Gamma_{ab}^{(\bm\beta)}=Q^{\bm\beta}\Gamma_{ab}Q^{-\bm\beta}.
\label{eq:45}
\end{equation}
Now clearly, $\Gamma_{ab}^{(\bm\beta)}$ still has degree
$\delta_{a}-\delta_{b}$ and $g^{(\bm\beta)}v_{0}$ is still a sum of
terms of degree a negative root.

On the other hand
$T^{\mathbf{k}}=T_{1}^{k_{1}}T_{2}^{k_{2}}\dots T_{n-1}^{k_{n-1}}$ has degree
$-\sum_{i=1}^{n-1}k_{i}\alpha_{1}$. Hence, if at least one of the $k_i$
is less than zero, then the degree of $T^{\mathbf{k}}$  contains a
positive root. This proves (by orthogonality of elements of distinct degree in
$F^{(n)}$) the following lemma.

\begin{lem}
  Let $\mathbf{k}=\sum_{i}k_{i}\alpha_{i}$ and suppose that for at
  least one $1\le i\le n-1$ we have   $k_{i}<0$ then
\[
\tau_{\mathbf{k}}^{(\bm\beta)}=\langle T^{\mathbf{k}}v_{0},g^{(\bm\beta)}v_{0}\rangle=0.
\]
\end{lem}

\subsection{Heine Formula for the $\tau$-functions.}
\label{sec:heine-formula-tau}

In this section we calculate the $\tau$-functions,
$\tau_{\mathbf{k}}^{(\bm\beta)}$, in the case that they are not zero.

Assume first that $\bm\beta=\bm0$, so consider $\tau_{\mathbf{k}}=\langle
T^{\mathbf{k}}v_{0},gv_{0}\rangle$. We start with the factorization ~\eqref{eq:37} of $g$ in
exponentials. Expanding the exponentials we see that
\begin{equation}
gv_{0}=\sum_{\mathbf{m}}\Gamma^{(\mathbf{m})}v_{0},\quad 
\Gamma^{(\mathbf{m})}=\left(\overrightarrow{\prod_{a>b}}\frac{(\Gamma_{ab})^{m_{ab}}}{m_{ab}!}\right),
\label{eq:42}
\end{equation}
where $\mathbf{m}=(m_{ab})$ is a triangular array of nonnegative integers
$m_{ab}\ge 0, 0\le b<a\le n-1$. Now the degree of
$\Gamma^{(\mathbf{m})}$ is
\[
\omega^{(\mathbf{m})}=\sum_{a>b}m_{ab}(\delta_{a}-\delta_{b})=-\sum_{i=1}^{n-1}M^{(\mathbf{m})}_{i}\alpha_{i},
\]
where, using \eqref{eq:38},
{
\[
M^{(\mathbf{m})}_{i}=\sum_{0\le \ell<i\le k\le n-1}m_{k\ell}.
\]}
Recalling that the degree of $T^{\mathbf{k}}$ is
$-\sum_{i=1}^{n-1}k_{i}\alpha_{i}$ we see that the only contributions
to the $\tau$-function $\tau_{\mathbf{k}}=\langle T^{\mathbf{k}}v_{0},gv_{0}\rangle$ are the terms
$\Gamma^{(\mathbf{m})}$ where $k_{i}=M^{(\mathbf{m})}_{i}$ for all
$1\le i\le n-1$. We write this condition as
$\mathbf{k}=\mathbf{M}^{(\mathbf{m})}$, where
$\mathbf{M}^{(\mathbf{m})}=(M^{(\mathbf{m})}_{1},
M^{(\mathbf{m})}_{2},\dots, M^{(\mathbf{m})}_{n-1})$. Therefore the
$\tau$-function is a
finite sum:
\[
\tau_{\mathbf{k}}=
\sum_{
\substack{\bm m\\ \mathbf{k}=\mathbf{M}^{(\bm m)}}}
\Gamma^{(\mathbf{m})}_{\mathbf{k}},\quad
\Gamma^{(\mathbf{m})}_{\mathbf{k}}=\langle T^{\mathbf{k}}v_{0},\Gamma^{(\mathbf{m})}v_{0}\rangle.
\]
So we need to calculate each of the terms $\Gamma^{(\mathbf{m})}_{\mathbf{k}}=\langle
T^{\mathbf{k}}v_{0},\Gamma^{(\mathbf{m})}v_{0}\rangle$. 

Recall that $\Gamma^{(\mathbf{m})}$ is a product of powers of the
$\Gamma_{ab}$,
$\Gamma_{ab}=\Res_{z,w}\left( C_{ab}(z,w)E_{ab}(z,w)\right), $ and
$E_{ab}(z,w)=\psi_{a}^{+}(z)\psi_{b}^{-}(w)$. The variables $z,w$ in
the definition of $\Gamma_{ab}$ are dummy variables, and it will be
useful to introduce separate variables $z^{(ab)}_{i},w^{(ab)}_{i}$,
$1\le i\le m_{ab}$ for each factor in $\Gamma^{m_{ab}}_{ab}$. So we
write
\[
(\Gamma_{ab})^{m_{ab}}=\prod_{i=1}^{m_{ab}}
\Res_{z^{(ab)}_{i},w^{(ab)}_{i}}\left(C_{ab}(z^{(ab)}_{i},w^{(ab)}_{i})E_{ab}(z^{(ab)}_{i},w^{(ab)}_{i})\right).
\]
Then
\[
\Gamma^{(\mathbf{m})}_{\mathbf{k}}=\Res_{\mathbf{z},\mathbf{w}}
\left(C^{(\mathbf{m})}(\mathbf{z},\mathbf{w})P^{(\mathbf{m})}_{\mathbf{k}}(\mathbf{z},\mathbf{w})\right),
\]
where $\mathbf{z},\mathbf{w}$ are the sets of all variables
$z^{(ab)}_{i}$, $w_{i}^{(ab)}$, 
 $a>b, 1\le i\le m_{ab}$, 
$\Res_{\mathbf{z},\mathbf{w}}$ is the residue over all the variables
in $\mathbf{z},\mathbf{w}$ and
\[
C^{(\mathbf{m})}(\mathbf{z},\mathbf{w})=\prod_{a>b}\prod_{i=1}^{m_{ab}}C_{ab}(z^{(ab)}_{i},w^{(ab)}_{i}),
\]
and
\begin{multline}
P^{(\mathbf{m})}_{\mathbf{k}}(\mathbf{z},\mathbf{w})=\left\langle
  T^{\mathbf{k}}v_{0},\overrightarrow{\prod_{a>b}}\frac{\prod_{i=1}^{m_{ab}}E_{ab}(z^{(ab)}_{i},w^{(ab)}_{i})v_{0}}{m_{ab}!}\right\rangle=\\
=
\left\langle T^{\mathbf{k}}v_{0},\overrightarrow{\prod_{a>b}}\frac{\prod_{i=1}^{m_{ab}}\psi^{+}_{a}(z^{(ab)}_{i})\psi_{b}^{-}(w^{(ab)}_{i})v_{0}}{m_{ab}!}\right\rangle.
\label{eq:43}
\end{multline}

The idea is now to reduce the term
$\langle T^{\mathbf{k}}v_{0},\Gamma^{(\mathbf{m})}v_{0}\rangle$ as a
product of matrix elements, each of which contains only
$Q_{a}^{K_{a}}$ on the left hand side, and fermion fields
$\psi^{{\pm}}_{a}(z)$ of the same type $a$ on the right hand side.

Rewrite the product of fermion fields in the  RHS of \eqref{eq:43} as
\begin{equation}
\overrightarrow{\prod_{a>b}}\prod_{i=1}^{m_{ab}}\psi^{+}_{a}(z^{(ab)}_{i})
\psi_{b}^{-}(w^{(ab)}_{i})=(-1)^{G(\mathbf{m})}\Psi_{0}\Psi_{1}\dots\Psi_{n-1},
\label{eq:74}
\end{equation}
where
\[
\Psi_{a}=\overrightarrow{\prod_{a>b}}\prod_{i=1}^{m_{ab}}\psi_{a}^{+}(z^{(ab)}_{i}) 
\overrightarrow{\prod_{c>a}}\prod_{i=1}^{m_{ca}}\psi_{a}^{-}(w^{(ca)}_{i}).
\]
Here we use the same ordering on fermion fields in the variables
$z^{(ab)}_{i}$ and $w^{(ca)}_{i}$ as in
\eqref{eq:44}. 

{To see that \eqref{eq:74} indeed holds, note that we
  need to reorder the fermion fields on the LHS in order to achieve
  the form of the RHS. In this process (because of the way we have defined the ordered product, see (\ref{eq:44})), we never need to move a
  positive fermion field, $\psi_{a}^{+}(z)$, past a negative fermion
  field, $\psi_{a}^{-}(w)$, of the same type $a$, and vice versa. This
  means that no delta function terms (see \eqref{eq:9}) show up, and
  the result of the reordering is just a sign denoted by
  $(-1)^{G(\mathbf{m})}$. To determine this sign observe that the same sign would appear if we replaced each fermion field $\psi_{a}^{\pm}(z)$ by the corresponding
  fermionic translation operator $Q_{a}^{\pm1}$ and reordered the resulting product in the
basis $Q^{\bm \beta}$ of \eqref{eq:69}. That is,
\[
\prod_{a>b}\prod_{i=1}^{m_{ab}}Q_{a}Q_{b}\inv=(-1)^{G(\mathbf{m})}Q^{\omega^{\mathbf{m}}},\quad \omega^{\mathbf{m}}=\sum_{a>b}m_{ab}(\delta_{a}-\delta_{b}),
\]
where $(-1)^{G(\mathbf{m})}$ is the same sign as the one appearing in (\ref{eq:74}). This means that we
can express the sign in terms of the cocycle $\epsilon$ of Lemma
\ref{lem:CentralExtLattice}. Since the formula does not seem
particularly enlightening, we leave this to the interested reader.
}

Now recall that
\[
T^{\mathbf{k}}=Q^{-\mathbf{k}}=Q_{0}^{K_{0}}Q_{1}^{K_{1}}\dots Q_{n-1}^{K_{n-1}},
\]
where $K_{a}=-\mathbf{k}\cdot \delta_{a}=k_{a}-k_{a+1}$ with $k_{0}=k_n=0$.  Using \eqref{eq:74} we
find, also using the factorization lemma of
multicomponent fermion fields \cite{taubirkhoff},
\[
  \begin{aligned}
    P^{(\mathbf{m})}_{\mathbf{k}}(\mathbf{z},\mathbf{w})&=
(-1)^{G(\mathbf{m})}
\left\langle  Q_{0}^{K_{0}}\dots
      Q_{n-1}^{K_{n-1}}v_{0},\Psi_{0}\dots\Psi_{n-1}v_{0}\right\rangle=\\&=
(-1)^{G(\mathbf{m})}\prod_{a=1}^{n-1}\left\langle
      Q_{a}^{K_{a}}v_{0},\Psi_{a}v_{0}\right\rangle.
  \end{aligned}
\]
Now, by a formula \cite{taubirkhoff} for the matrix elements of one-component
fermions, we have
\begin{equation}
\left\langle
      Q_{a}^{K_{a}}v_{0},\Psi_{a}v_{0}\right\rangle=
\prod\left(z^{(ab)}_{i}-z^{(ab^{\prime})}_{j}\right) 
\prod\left( w^{(ba)}_{i}-w^{(b^{\prime}a)}_{j}\right) 
\prod\left( z^{(ab)}_{i}-w^{(ca)}_{j}\right)\inv
\label{eq:88}
\end{equation}
Here, the product $\prod\left( z^{(ab)}_{i}-w^{(ca)}_{j}\right)^{-1}$ is over all $b,c,i,j$ with $c>a$, $b<a$ and $1\le i\le m_{ab}$ and $1\le j\le m_{ca}$. The product $\prod\left(z^{(ab)}_{i}-z^{(ab^{\prime})}_{j}\right)$ is over all $b,b^{\prime}, i,j$ with $b<b^{\prime}$, and $1\le i\le m_{ab}$ and $1\le j\le m_{ab^\prime}$ and all $b,b^\prime, i, j$ with $b=b^{\prime}$ and $1\le i<j\le m_{ab}$. Similarly, the product $\prod\left( w^{(ba)}_{i}-w^{(b^{\prime}a)}_{j}\right)$ is over all $b,b^\prime, i, j$ with $b< b^{\prime}$, $1\le i\le m_{ba}$ and $1\le j\le m_{b^\prime a}$, and $b=b^\prime$ with $1\le i<j\le m_{ba}$. (Here, we expand $\left( z^{(ab)}_{i}-w^{(ca)}_{j}\right)^{-1}$ in positive powers of $w^{(ca)}_{j}$.)
Putting this all together gives us the following theorem.
\begin{thm} \label{thm:tauk}
  \[
\tau_{\mathbf{k}}=\sum_{\mathbf{m}\atop \mathbf{M}^{\mathbf{m}}=\mathbf{k}}\Gamma^{(\mathbf{m})}_{\mathbf{k}},
\]
where
\begin{multline*}
  \Gamma^{(\mathbf{m})}_{\mathbf{k}}=
(-1)^{G(\mathbf{m})}
\Res_{\mathbf{z},\mathbf{w}}
\left[\prod_{a>b}\frac{\prod_{i=1}^{m_{ab}}
  C_{ab}(z^{(ab)}_{i},w^{(ab)}_{i})}{m_{ab}!} \right.\times\\
\times\left.\left(\prod\left(
      z^{(ab)}_{i}-z^{(ab^{\prime})}_{j}\right) \prod\left(
      w^{(ba)}_{i}-w^{(b^{\prime}a)}_{j}\right) \prod\left(
      z^{(ab)}_{i}-w^{(ca)}_{j}\right)\inv\right)\right].
\end{multline*}
\end{thm}
Finally we turn to $\tau_{\mathbf{k}}^{(\bm\beta)}=\langle
T^{\mathbf{k}}v_{0},g^{(\bm\beta)}v_{0}\rangle$. Recall that
$g^{(\bm\beta)}=Q^{\bm\beta}g(Q^{\bm\beta})\inv$, see~\eqref{eq:45}. Using an
expansion similar to \eqref{eq:42} we get
\[
\tau_{\mathbf{k}}^{(\bm\beta)}=\sum_{\mathbf{m}\atop
  \mathbf{k}=\mathbf{M}^{(\mathbf{m})}}\Gamma^{(\bm\beta,\mathbf{m})}_{\mathbf{k}},\quad
\Gamma^{(\bm\beta,\mathbf{m})}_{\mathbf{k}}=\langle T^{\mathbf{k}}v_{0},\Gamma^{(\bm\beta,\mathbf{m})}v_{0}\rangle,
\]
where
\[
\Gamma^{(\bm\beta,\mathbf{m})}=\left(\overrightarrow{\prod_{a>b}}
\frac{\left(\Gamma^{(\bm\beta)}_{ab}\right)^{m_{ab}}}{m_{ab}!}\right),\quad
\Gamma^{(\bm\beta)}_{ab}=Q^{\bm\beta}\Gamma_{ab}(Q^{\bm\beta})\inv.
\]
If $\bm\beta=(\beta_{0},\beta_{1},\dots,\beta_{n-1})$ write
$\abs{\beta}=\sum_{a=0}^{n-1}\beta_{a}$. (This is the total degree of
$\bm\beta$, see section \ref{sec:ferm-transl-oper}.) Also, if
$\bm\gamma=(\gamma_{0},\dots,\gamma_{n-1})$ define
$\bm\beta\cdot\bm\gamma=\sum_{a=0}^{n-1}\beta_{a}\gamma_{a}$,
\begin{lem}
  \begin{itemize}
  \item
    $Q^{\bm\beta}\psi_{a}^{\pm}(z)(Q^{\bm\beta})\inv=(-1)^{\abs{\bm\beta-\beta_{a}\delta_{a}}}z^{{\mp\bm\beta\cdot\delta_{a}}}\psi_{a}^{\pm}(z)$.
  \item
    $Q^{\bm\beta}E_{ab}(z,w)(Q^{\bm\beta})\inv=(-1)^{\bm\beta\cdot(\delta_{a}-\delta_{b})}
z^{-\bm\beta\cdot\delta_{a}}
w^{\bm\beta\cdot    \delta_{b}}
E_{ab}(z,w)$.
  \end{itemize}
\end{lem}
Define
\[
C_{ab}^{(-\beta_{a},\beta_{b})}(z,w)=\sum_{k,\ell\in\mathbb{Z}} C^{(a,b)}_{k-\beta_{a},\ell+\beta_{b}}z^{-k-1}w^{-\ell-1}=z^{{-}\beta_{a}}w^{\beta_{b}}C_{ab}(z,w).
\]
Then
\begin{lem}
  \[
\Gamma^{(\bm\beta)}_{ab}=(-1)^{\bm\beta\cdot(\delta_{a}-\delta_{b})}\Res_{z,w} \left(C^{(-\beta_{a},\beta_{b})}_{ab}(z,w)E_{ab}(z,w)\right).
\]
\end{lem}
The upshot is that the formula for
$\tau_{\mathbf{k}}^{(\bm\beta)}$ is the same as the one in Theorem
\ref{thm:tauk}, but with $C_{ab}(z^{(ab)}_{i},w^{(ab)}_{i})$ replaced
by $(-1)^{\bm\beta\cdot(\delta_{a}-\delta_{b})} C_{ab}^{(-\beta_{a},\beta_{b})}(z^{(ab)}_{i},w^{(ab)}_{i})$.

\section{Shift Fields and the Formula for the Negative Component of the Gauss Factorization.}
\label{sec:shift-fields-formula}

Recall that the Baker function was given by
\[
\Psi_{\mathbf{k}}^{(\bm\beta)}=\overline {T}^{\,\mathbf{k}}\overline Q^{\,-\bm\beta}g_{-}^{(\mathbf{k},\bm\beta)},
\]
where $g_{-}^{(\mathbf{k},\bm\beta)}$ is the minus component of the Gauss
  factorization of $g^{(\mathbf{k},\bm\beta)}$, see \eqref{eq:46}. In
  this section we give a formula for $g_{-}^{(\mathbf{k},\bm\beta)}$ in
  terms of ``shift fields''. To explain what these are, recall the
  coordinates $c^{a,b}_{k,\ell}$ on the lower triangular subgroup
  $\mathcal{N}$, see \eqref{eq:11}. Then define multiplicative maps on
  these coordinates by
\[
S_{ab}^{\alpha,\beta}c^{a^{\prime},b^{\prime}}_{k,\ell}=
\delta_{aa^{\prime}}\delta_{bb^{\prime}}c^{a,b}_{k+\alpha,\ell+\beta},
\quad 
S_{ab}^{\alpha,\beta}(1)=0.
\]
Then define \emph{shift fields}
\[
S_{ab}^{\pm1,0}(z)=\left(1-\frac{S_{ab}^{1,0}}{z}\right)^{\pm 1},\quad
S_{ab}^{0,\pm1}(w)=\left(1-\frac
{S^{0,1}_{ab}}{w}\right)^{\pm1}.
\]
Here we expand in positive powers of $1/z$ and $1/w$.

The shift fields act on the generating series $C_{ab}(z,w)$ of \eqref{eq:83} by
\begin{equation}
  \label{eq:84}
  S^{\pm1,0}_{ab}(z) C_{cd}(z_{1},w_{1})=
      (1-\frac{z_{1}}{z})^{\pm\delta_{ac}\delta_{bd} } C_{cd}(z_{1},w_{1}),
\end{equation}
and
\begin{equation}
\label{eq:85}
  S^{0,\pm1}_{ab}(w) C_{cd}(z_{1},w_{1})=
    (1-\frac{w_{1}}{w})^{\pm \delta_{ac}\delta_{bd}} C_{cd}(z_{1},w_{1}).
\end{equation}
Finally define
\[
S_{a}^{\pm}(z)=\prod_{a>b}S^{\pm1,0}_{ab}(z)\prod_{b>a}S_{ba}^{0,\mp 1}(z).
\]
Then
\begin{equation}
  \label{eq:86}
  S^{\pm}_{a}(z)C_{cd}(z_{1},w_{1})=(1-z_{1}/z)^{\pm\delta_{ac}}(1-w_{1}/z)^{\mp\delta_{ad}}C_{cd}(z_{1},w_{1}).
\end{equation}
Indeed, if $a\ne c,d$ the RHS of (13.3) is just $C_{cd}(z_{1},w_{1})$. In the case $a=c$ we have
\begin{equation*}
  z^{\pm 1}S^{\pm}_{a}(z)C_{cd}(z_{1},w_{1})=
  z^{\pm 1}S^{\pm}_{a}(z)C_{ad}(z_{1},w_{1})=
  z^{\pm 1}S^{\pm1,0}_{ad}(z)C_{ad}(z_{1},w_{1})=
  z^{\pm 1}(1-z_{1}/z)^{\pm1}C_{ad}(z_{1},w_{1}),
\end{equation*}
by \eqref{eq:84}. Similarly, if $a=d$, we have
\begin{equation*}
  z^{\mp 1}S^{\pm}_{a}(z)C_{cd}(z_{1},w_{1})=
  z^{\mp 1}S^{\pm}_{a}(z)C_{ca}(z_{1},w_{1})=
  z^{\mp 1}S^{0,\mp1}_{ca}(z)C_{ca}(z_{1},w_{1})=
  z^{\mp 1}(1-w_{1}/z)^{\mp 1}C_{ca}(z_{1},w_{1}),
\end{equation*}
by \eqref{eq:85}.

The factors $(z-z_{1})^{\pm\delta_{ac}}(z-w_{1})^{\mp\delta_{ad}}$  also show up in certain matrix
elements, involving insertions of fermion fields.

\begin{lem}
  \label{lem:FermionEmatrixelmnt} For $c>d$
  \begin{equation*}
    \langle Q_{a}^{\pm1}T^{\delta_{d}-\delta_{c}}v_{0},
    \psi_{a}^{\pm}(z)E_{cd}(z_{1},w_{1})v_{0}\rangle=-(z-z_{1})^{\pm\delta_{ac}}(z-w_{1})^{\mp\delta_{ad}}
    \langle T^{\delta_{d}-\delta_{c}}v_{0},
    E_{cd}(z_{1},w_{1})v_{0}\rangle,
\end{equation*}
 where the $E_{cd}(z_{1},w_{1})$ are  the generating series \eqref{eq:18}.
\end{lem}

The proof of the lemma is a calculation of matrix elements of fermion
fields.

The generating series $C_{cd}(z,w)$ appears in
$\tau$-functions and other matrix elements together with the Lie
algebraic generating series $E_{cd}(z,w)$ in the combination
$\Gamma_{cd}=\Res_{z,w}\left(C_{cd}(z,w)E_{cd}(z,w)\right)$, see
\eqref{eq:82}.

\begin{cor}\label{corShiftFermionInsert} If $c>d$
  \[
\langle Q_{a}^{\pm1}T^{\delta_{d}-\delta_{c}}v_{0},
\psi_{a}^{\pm}(z)\Gamma_{cd}v_{0}\rangle=-z^{\pm \delta_{ac}}z^{\mp \delta_{ad}}S_{a}^{\pm}(z)
\langle T^{\delta_{d}-\delta_{c}}v_{0}, \Gamma_{cd}v_{0}\rangle.
  \]
\end{cor}

The proof of the Corollary follows from the combination of
\eqref{eq:86} and Lemma \ref{lem:FermionEmatrixelmnt}.

So the upshot is that the shift fields ``correspond'' to insertion of
fermion fields.

\begin{thm} \label{thm:gminus}
Let $g_{-}^{(\mathbf{k},\bm\beta)}(z,w)_{ab}=\Res_{z,w}\left(g_{ab}(z,w)E_{ab}(z,w)\right)$, where $g_{ab}(z,w)$ is the series in Lemma \ref{lem1} in the case that $g=g^{(\mathbf{k},\bm{\beta})}$. For $a\ne b$
  \[
g_{-}^{(\mathbf{k},\bm\beta)}(z,w)_{ab}=(-1)^{\mathbf{k}\cdot(\delta_{a}-\delta_{b})}\epsilon(\delta_{a}-\delta_{b},\mathbf{k})\epsilon(\delta_a,\delta_b)
\frac{S_{a}^{-}(z)S_{b}^{+}(w)\tau_{\mathbf{k}+\delta_{b}-\delta_{a}}^{(\bm\beta)}}{zw\tau_{\mathbf{k}}^{(\bm\beta)}}.
\]
\end{thm}
\begin{proof} We use Lemma \ref{lem1} to calculate $g_{ab}(z,w)$ for $g=g^{(\mathbf{k},\bm{\beta})}$ in the case where $a\ne b$.
  \begin{align*}
g_{ab}(z,w)&=\langle
             v_{0},\psi_{b}^{+}(w)\psi_{a}^{-}(z)(T^{\mathbf{k}})\inv g^{(\bm\beta)}v_{0}\rangle/\tau_{\mathbf{k}}^{(\bm\beta)}=\\
           &=\epsilon(\mathbf{k},-\mathbf{k})\langle
             v_{0},\psi_{b}^{+}(w)\psi_{a}^{-}(z)Q^{\mathbf{k}} g^{(\bm\beta)}v_{0}\rangle/\tau_{\mathbf{k}}^{(\bm\beta)},
  \end{align*}
since $(T^{\mathbf{k}})\inv=\epsilon(\mathbf{k},-\mathbf{k})T^{-\mathbf{k}}=\epsilon(\mathbf{k},-\mathbf{k})Q^{\mathbf{k}}$, see Lemma
\ref{Lem:GeneratorsTranslationGroupIdentities}, Part (2). Now 
\[
\psi_{b}^{+}(w)\psi_{a}^{-}(z) Q_{c}^{k}=
\begin{cases}
  (-z)^{-k}Q_{c}^{k}\psi_{b}^{+}(w)\psi_{a}^{-}(z)
  &\text{ if $c=a$}\\
(-w)^{k}Q_{c}^{k}\psi_{b}^{+}(w)\psi_{a}^{-}(z)
  &\text{ if $c=b$}\\
Q_{c}^{k}\psi_{b}^{+}(w)\psi_{a}^{-}(z)
  &\text{ if $c\ne a$ and $c\ne b$}.
\end{cases}
\]
So overall we have
\[
\psi_{b}^{+}(w)\psi_{a}^{-}(z)Q^{\mathbf{k}}=
(-z)^{-\mathbf{k}\cdot \delta_{a}}(-w)^{\mathbf{k}\cdot \delta_{b}}
Q^{\mathbf{k}}\psi_{b}^{+}(w)\psi_{a}^{-}(z).
\]
Hence
\begin{align*}
g_{ab}(z,w)&=
(-z)^{-\mathbf{k}\cdot \delta_{a}}(-w)^{\mathbf{k}\cdot \delta_{b}}  \epsilon(\mathbf{k}, -\mathbf{k})
\langle
v_{0},Q^{\mathbf{k}}\psi_{b}^{+}(w)\psi_{a}^{-}(z)g^{(\bm\beta)}v_{0}\rangle/\tau_{\mathbf{k}}^{(\bm\beta)}=\\
&=(-z)^{-\mathbf{k}\cdot \delta_{a}}(-w)^{\mathbf{k}\cdot \delta_{b}}
\langle T^{\mathbf{k}}v_{0},\psi_{b}^{+}(w)\psi_{a}^{-}(z)g^{(\bm\beta)}v_{0}\rangle/\tau_{\mathbf{k}}^{(\bm\beta)}.
\end{align*}
Now \begin{multline*}T^{\mathbf{k}}=Q_{b}Q_{a}\inv Q_{a}Q_{b}\inv T^{\mathbf{k}}=Q_{b}Q_{a}\inv Q_{a}Q_{b}\inv Q^{-\mathbf{k}}=Q_b Q_a \inv Q_a \epsilon(-\delta_b,-\mathbf{k})Q^{-\mathbf{k}-\delta_b}=\\=Q_b Q_a \inv \epsilon(-\delta_b,-\mathbf{k})\epsilon(\delta_a,-\mathbf{k}-\delta_b)Q^{-\mathbf{k}-\delta_b+\delta_a}=\epsilon(\delta_a-\delta_b,\mathbf{k})\epsilon(\delta_a,\delta_b)Q_bQ_a\inv T^{\mathbf{k}+\delta_b-\delta_a},\end{multline*} by
Lemma \ref{Lem:GeneratorsTranslationGroupIdentities} and Lemma \ref{lem:propcocycle}, so 
\begin{align*}
g_{ab}(z,w)&=
(-z)^{-\mathbf{k}\cdot \delta_{a}}(-w)^{\mathbf{k}\cdot \delta_{b}}
\epsilon(\delta_a-\delta_b,\mathbf{k})\epsilon(\delta_a,\delta_b)\times\\
&\qquad\qquad\times\langle Q_{b}Q_{a}\inv T^{\mathbf{k}+\delta_{b}-\delta_{a}}
v_{0},\psi_{b}^{+}(w)\psi_{a}^{-}(z)g^{(\bm\beta)}v_{0}\rangle/\tau_{\mathbf{k}}^{(\bm\beta)}=\\
&=(-1)^{\mathbf{k}\cdot(\delta_{b}-\delta_{a})}
\epsilon(\delta_a-\delta_b,\mathbf{k})\epsilon(\delta_a,\delta_b)\times\\
&\qquad\qquad\times S_{b}^{+}(w)S_{a}^{-}(z)
\langle  T^{\mathbf{k}+\delta_{b}-\delta_{a}}v_{0},g^{(\bm\beta)}v_{0}\rangle
/zw\tau_{\mathbf{k}}^{(\bm\beta)}=\\
&=(-1)^{\mathbf{k}\cdot(\delta_{b}-\delta_{a})}
\epsilon(\delta_a-\delta_b,\mathbf{k})\epsilon(\delta_a,\delta_b)
S_{b}^{+}(w)S_{a}^{-}(z)
\tau_{\mathbf{k}+\delta_{b}-\delta_{a}}^{(\bm \beta)}
/zw\tau_{\mathbf{k}}^{(\bm\beta)},
\end{align*}
by applying Corollary \ref{corShiftFermionInsert}.

\end{proof}

\section{The Short Relations.}
\label{sec:2+k-term-relations}

In this section we are going to derive bilinear equations for the
$\tau$-functions of type $nT$. These equations will have $2+k$ terms,
where $1\le k\le n-2$. We will refer to these as the short relations
of type $nT$, as they have up to $n$ terms. In the next section we
will discuss relations of length $n+1$.

Pick a $k$-element subset $I\subset \{0,1,\dots,n-1\}$ and define
\[
\delta_{I}=\sum_{i\in I}\delta_{i}.
\]
Next, given $I$, pick an element
$\delta_{J}=\sum_{a=0}^{n-1}j_{a}\delta_{a}\in\mathbb{Z}^{n}$, such
that $j_{a}\ge0$ and $\abs{\delta_{j}}=\sum_{a=0}^{n-1}j_{a}=k$. Then
define a root $\rho\in A_{n-1}$ by $\rho=\delta_{I}-\delta_{J}$. Then
we have
\begin{equation}
\overline T^{\,\rho}\overline Q^{-\delta_{J}}=\overline Q^{-\delta_{I}}.
\label{eq:81}
\end{equation}

For instance, for any $I$ we can take $\delta_{J}=\delta_{I}$, so that
$\rho=0$. But there are other possibilities, e.g.,
$\overline T_{1}\overline Q_{1}^{-1}=\overline Q_{0}^{-1}$, so that for $n=2$, $k=1$
and $\delta_{I}=\delta_{0}$ we can take $\delta_{J}=\delta_{1}$, and $\rho=\delta_{0}-\delta_{1}=\alpha_{1}$.

Consider the connection matrix
\[
\Gamma^{(\mathbf{k}+\rho,\bm\beta+\delta_{J})}_{(\mathbf{k},\bm\beta)}=
\left(g_{-}^{(\mathbf{k},\bm\beta)}\right)\inv
\left(\overline
Q^{\,-\delta_{I}}\right)\left(g_{-}^{(\mathbf{k}+\rho,\bm\beta+\delta_{J})}\right).
\]
Denote by $\gamma_{bc}^{i,j (\mathbf{k},\bm\beta)}$ the coefficient of
$E_{bc}^{-i-1,j}$ in $g_{-}^{(\mathbf{k},\bm\beta)}$ and put 
$\gamma_{bc}^{(\mathbf{k},\bm\beta)}=\gamma_{bc}^{0,0
  (\mathbf{k},\bm\beta)}$. 
Using \eqref{eq:20} we find 
\begin{multline}\label{eq:51}
\Gamma^{(\mathbf{k}+\rho,\bm\beta+\delta_{J})}_{(\mathbf{k},\bm\beta)}=
\left(
    1-\sum_{i,j\ge0}\sum_{b,c}\gamma_{bc}^{i,j
      (\mathbf{k},\bm\beta)}E_{bc}^{-i-1,j} + \dots
\right) \times\\
\times
(-1)^{k-1}\left(
  \sum_{s\in\mathbb{Z}}
  \left(
    \sum_{i\in I}E_{ii}^{s+1,s}-\sum_{d\notin      I}E_{dd}^{s,s})
  \right)
\right) \times\\ \times
\left(
    1+\sum_{\ell,m\ge0}\sum_{e,f}\gamma_{ef}^{\ell,m
      (\mathbf{k}+\rho,\bm\beta+\delta_{J})}E_{ef}^{-\ell-1,m} \right).
\end{multline}
Now we can write, see \eqref{eq:12},
\[
\Gamma^{(\mathbf{k}+\rho,\bm\beta+\delta_{J})}_{(\mathbf{k},\bm\beta)}=
g_{0+}^{(\mathbf{k},\bm\beta)}
\left(
  \overline Q^{\,-\delta_{J}}
\right)
\left(
g_{0+}^{(\mathbf{k}+\rho,\bm\beta+\delta_{J})}
\right)\inv,
\]
which shows, because all $j_{s}\ge0$,  that the connection matrix is
 nonnegative, in particular,
the coefficient of $E_{bc}^{-1,0}$ is zero. Using
$E_{ab}^{k,\ell}E_{cd}^{m,n}=E_{a,d}^{k,n}\delta_{bc}\delta_{\ell m}$ we
find from \eqref{eq:51} for $b,c\notin I$:
\begin{equation}
\gamma_{bc}^{(\mathbf{k},\bm\beta)}-\gamma_{bc}^{(\mathbf{k}+\rho,\bm\beta+\delta_{J})}-
\sum_{i\in I}\gamma_{bi}^{(\mathbf{k},\bm\beta)}\gamma_{ic}^{(\mathbf{k}+\rho,\bm\beta+\delta_{J})}=0.
\label{eq:52}
\end{equation}
Now $\gamma_{bc}^{(\mathbf{k},\bm\beta)}$ is the coefficient of
$z\inv w\inv$ of $g_{bc}^{(\mathbf{k},\bm\beta)}(z,w)$ calculated in
Theorem \ref{thm:gminus}. Hence, if $b\ne c$,
\begin{equation}
\gamma_{bc}^{(\mathbf{k},\bm\beta)}=(-1)^{\mathbf{k}\cdot(\delta_{b}-\delta_{c})}
\epsilon(\delta_b,\delta_c)\epsilon(\delta_{b}-\delta_{c},\mathbf{k})
\frac
{\tau_{\mathbf{k}+\delta_{c}-\delta_{b}}^{(\bm\beta)}}
{\tau_{\mathbf{k}}^{(\bm\beta)}},
\label{eq:53}
\end{equation}
and this gives in 
\eqref{eq:52}
\begin{multline*}
  (-1)^{\mathbf{k}\cdot (\delta_{b}-\delta_{c})} 
\epsilon(\delta_b,\delta_c)\epsilon(\delta_{b}-\delta_{c},\mathbf{k})
\left[\frac{\tau_{\mathbf{k}+\delta_{c}-\delta_{b}}^{(\bm\beta)}}
    {\tau_{\mathbf{k}}^{(\bm\beta)}}- 
    (-1)^{\rho\cdot(\delta_{b}-\delta_{c})}
    \epsilon(\delta_{b}-\delta_{c},\rho)
    \frac{\tau_{\mathbf{k}+\rho+\delta_{c}-\delta_{b}}^{(\bm\beta+\delta_{J})}}
    {\tau_{\mathbf{k}+\rho}^{(\bm\beta+\delta_{J})}} 
    \right.\\
\left. 
    -\sum_{i\in I}
    (-1)^{\rho\cdot(\delta_{i}-\delta_{c})}\epsilon(\delta_b,\delta_c)\epsilon(\delta_b,\delta_i)\epsilon(\delta_i,\delta_c)
    \epsilon(\delta_{i}-\delta_{c},\rho)
    \frac{\tau_{\mathbf{k}+\delta_{i}-\delta_{b}}^{(\bm\beta)}}{\tau_{\mathbf{k}}^{(\bm\beta)}}
    \frac{\tau_{\mathbf{k}+\rho+\delta_{c}-\delta_{i}}^{(\bm\beta+\delta_{J})}}{\tau_{\mathbf{k}+\rho}^{(\bm\beta+\delta_{J})}}
  \right]=0.
\end{multline*}
Bringing all terms under a common denominator we have\\[.1cm]
\fbox{
\begin{minipage}{0.97\linewidth}
{
\begin{thm}[Short Relations]\label{thm:short-relations}
For all $k$ element subsets $I$ of
$\{0,1,\dots,n-1\}$, where $1\le k\le n-2$, pick
$\delta_{J}=\sum_{a=0}^{n-1}\delta_{a}\in \mathbb{Z}^{n}$ such that 
$j_{a}\ge 0$, $\sum_{a=0}^{n-1}j_{a}=k$. Let
$\rho=\delta_{I}-\delta_{J}$. Then, for all $b,c\notin I, b\ne c$,
\begin{multline*}
  \tau_{\mathbf{k}+\delta_{c}-\delta_{b}}^{(\bm\beta)}\tau_{\mathbf{k}+\rho}^{(\bm\beta+\delta_{J})}-
 (-1)^{\rho\cdot(\delta_{b}-\delta_{c})}
    \epsilon(\delta_{b}-\delta_{c},\rho,)
  \tau_{\mathbf{k}}^{(\bm\beta)}\tau_{\mathbf{k}+\rho+\delta_{c}-\delta_{b}}^{(\bm\beta+\delta_{J})}
  \\
  - \sum_{i\in  I}
    (-1)^{\rho\cdot(\delta_{i}-\delta_{c})}
    \epsilon(\delta_{i}-\delta_{c},\rho)\epsilon(\delta_b,\delta_c)\epsilon(\delta_b,\delta_i)\epsilon(\delta_i,\delta_c)
\tau_{\mathbf{k}+\delta_{i}-\delta_{b}}^{(\bm\beta)}\tau_{\mathbf{k}+\rho+\delta_{c}-\delta_{i}}^{(\bm\beta+\delta_{J})}
  =0.
\end{multline*}
  Here $\epsilon$ is the cocycle \eqref{eq:71} for $\mathbb{Z}^{n}$.
\end{thm}
}
\end{minipage}}
\ \\[.2cm]

 \section{The Long Relations.}
\label{sec:long-relations}

In the previous section we showed that $\tau$-functions of type $nT$
satisfy bilinear equations of length $2+k$, where $k=1,2,n-2$. So
these equations are of length $3,4,\dots, n$.  In particular, for $n=2$
this construction does not give any equations at all! However, we
claim that, for all $n$, there are also equations of length $n+1$, so there are
three term relations for $2T$ $\tau$-functions. We will discuss these
equations in this section.

To obtain these ``long'' equations, pick a nonempty subset
$I\subset\{0,1,\dots,n-1\}$ (of cardinality less than $n$) and let $J$ be the complement of $I$
in $\{0,1,\dots,n-1\}$. As in the previous section,
$\delta_{I}=\sum_{a\in I}\delta_{a}$ and
$\delta_{J}=\sum_{b\in J}\delta_{b}$.  Consider the
connection matrix
\[
\Gamma^{(\mathbf{k},\bm\beta+\delta_{ J})}_{(\mathbf{k},\bm\beta +\delta_{I})}.
\]
By \eqref{eq:59} and \eqref{eq:12} we have
\begin{equation}
  \label{eq:62}
  \begin{aligned}
    \Gamma^{(\mathbf{k},\bm\beta
      +\delta_{ J})}_{(\mathbf{k},\bm\beta+\delta_{I})}
&=\left(g_{-}^{(\mathbf{k},\bm\beta+\delta_{I})}\right)\inv
\left(
\overline  Q^{\,\delta_{I}-\delta_{ J}}
\right)
\left(g_{-}^{(\mathbf{k},\bm\beta+\delta_{ J})}\right)=\\
&=\left(g_{0+}^{(\mathbf{k},\bm\beta+\delta_{I})}\right)
\left(\overline
  Q^{\,\delta_{I}-\delta_{J}}\right)\left(g_{0+}^{(\mathbf{k},\bm\beta+\delta_{ J})}\right)\inv.
  \end{aligned}
\end{equation}
Note that, in contrast with the connection matrices in the
previous section,
$\Gamma^{(\mathbf{k},\bm\beta+\delta_{ J})}_{(\mathbf{k},\bm\beta
  +\delta_{I})}$ is no longer nonnegative. In particular, for any $a\in I$, the coefficient of $E_{aa}^{-1,0}$ is in general, nonzero. Calculating this coefficient in two ways using \eqref{eq:62}
will give us the desired length $n+1$ bilinear equations.

The first equality in \eqref{eq:62} gives us, similar to
\eqref{eq:51},
\begin{multline}\label{eq:63}
\Gamma^{(\mathbf{k},\bm\beta+\delta_{ J})}_{(\mathbf{k},\bm\beta+\delta_{I})}=
\left(
    1-\sum_{i,j\ge0}\sum_{b,c}\gamma_{bc}^{i,j
      (\mathbf{k},\bm\beta+\delta_{I})}E_{bc}^{-i-1,j} 
\right) \times
\\
\times
\left( (-1)^{n-1}
  \sum_{k\in\mathbb{Z}}
  \left(
    \sum_{b\in J}E_{bb}^{k+1,k}+\sum_{a\in I}E_{aa}^{k-1,k})
  \right)
\right) \times\\ \times
\left(
    1+\sum_{\ell,m\ge0}\sum_{e,f}\gamma_{ef}^{\ell,m
      (\mathbf{k},\bm\beta+\delta_{ J})}E_{ef}^{-\ell-1,m} \right).
\end{multline}
Then the coefficient $X_{aa}$ of $E_{aa}^{-1,0}$ is 
\begin{equation}
  X_{aa}=(-1)^{n-1}
  \left(
    1 -\sum_{b\in  J}
    \gamma_{ab}^{(\mathbf{k},\bm\beta+\delta_{I})}
    \gamma_{ba}^{(\mathbf{k},\bm\beta+\delta_{ J})}
  \right).
\label{eq:64}
\end{equation}
Using \eqref{eq:53} we can write this as 
\begin{equation}
  \label{eq:76}
  X_{aa}=(-1)^{n-1}
  \left(
    1 +
    \sum_{b\in  J}
    \frac{\tau_{\mathbf{k}+\delta_{b}-\delta_{a}}^{(\bm\beta+\delta_{I})}}
    {\tau_{\mathbf{k}}^{(\bm\beta+\delta_{I})}}
    \frac{\tau_{\mathbf{k}+\delta_{a}-\delta_{b}}^{(\bm\beta+\delta_{J})}}
    {\tau_{\mathbf{k}}^{(\bm\beta+\delta_{J})}}
  \right).
\end{equation} 
(Here, we have used 
\begin{equation*}(-1)^{\mathbf{k}\cdot(\delta_a-\delta_b)}\epsilon(\delta_a,\delta_b)\epsilon(\delta_a-\delta_b,\mathbf{k})(-1)^{\mathbf{k}\cdot (\delta_b-\delta_a)}\epsilon(\delta_b,\delta_a)\epsilon(\delta_b-\delta_a,\mathbf{k})=\\ \epsilon(\delta_a,\delta_b)\epsilon(\delta_b,\delta_a)=-1,\end{equation*} since $a\ne b$, so either $\epsilon(\delta_a,\delta_b)=-1$ and $\epsilon(\delta_b,\delta_a)=1$ or vice versa.)

On the other hand, we can use the second expression for the connection
matrix in \eqref{eq:62} to calculate the coefficient of
$E_{aa}^{-1,0}$. Write
\[
g_{0+}^{(\mathbf{k},\bm\beta)}=\sum_{b,c}\sum_{i,j}\bar\gamma_{bc}^{i,j
(\mathbf{k},\bm\beta)}E_{bc}^{i,j},\quad
(g_{0+}^{(\mathbf{k},\bm\beta)})\inv=\sum_{b,c}\sum_{i,j}\tilde\gamma_{bc}^{i,j
(\mathbf{k},\bm\beta)}E_{bc}^{i,j}
\]
Then we have
\begin{equation}
X_{aa}=(-1)^{n-1}
\left(
\sum_{a^{\prime\prime}\in I}
\bar\gamma^{-1,-1(\mathbf{k},\bm\beta+\delta_{I})}_{aa^{\prime\prime}}\tilde\gamma^{0,0(\mathbf{k},\bm\beta+\delta_{J})}_{a^{\prime\prime}a^{\prime}}
\right).
\label{eq:65}
\end{equation}
We will now express (some of) the $\overline\gamma,\tilde\gamma$ variables in terms of matrix elements, which we will then use to write $X_{aa}$ entirely in terms of $\tau$-functions.

\begin{lem}
  \label{lem:g0+I}
  \begin{align}
 \tag{{A}}  \label{eq:66}
    \overline\gamma^{\,-1,-1(\mathbf{k},\bm\beta)}_{aa^{\prime}}&=
       \frac{
             \langle Q_{a}v_{0},
              g^{(\mathbf{k},\bm\beta)}Q_{a^{\prime}}v_{0}
               \rangle}
            {\tau_{\mathbf{k}}^{(\bm\beta)}},\\
  \tag{{B}}  \label{eq:67}
\tilde\gamma^{0,0(\mathbf{k},\bm\beta)}_{aa^{\prime}}&=
\frac{\langle Q_{a^\prime}\inv v_{0},
              g^{(\mathbf{k},\bm\beta)}Q_{a}\inv v_{0}
               \rangle}
            {\tau_{\mathbf{k}}^{(\bm\beta)}}
  \end{align}
\end{lem}

\begin{proof}
  For simplicity let us start by ignoring
  for the moment the dependence on $(\mathbf{k},\bm\beta)$. So let
  $g_{0+}$ be the nonnegative component in the Gauss factorization of
  an infinite matrix $g=g_{-}g_{0+}$, where $g$ belongs to
  $GL_{\infty}^{(n)}$ (or some completion). So $g_{0+}$ is an infinite
  matrix with entries $n\times n$ blocks, $\overline\gamma,\tilde\gamma$:
\[
g_{0+}=\sum_{i,j\ge0}\bar
  \gamma_{ij}E_{ij}+\sum_{k\in\mathbb{Z}\atop \ell<0}\bar \gamma_{k\ell}E_{k\ell},
\]
where $\bar\gamma_{ij}=\sum_{bc}\bar\gamma_{bc}^{i,j}E_{bc}$.
Pictorially we have
\begin{equation}
g_{0+}=
\begin{bmatrix}
  \ddots&\vdots&\vdots&\vdots&\vdots&\vdots&\dots\\
  \dots    &\bar\gamma_{1,2}&\bar\gamma_{1,1}&\bar\gamma_{1,0}&\bar\gamma_{1,-1}&\bar\gamma_{1,-2}&\dots\\
  \dots    &\bar\gamma_{0,2}&\bar\gamma_{0,1}&\bar\gamma_{0,0}&\bar\gamma_{0,-1}&\bar\gamma_{0,-2}&\dots\\
  \dots    &0&0&0&\bar\gamma_{-1,-1}&\bar\gamma_{-1,-2}&\dots\\
  \dots    &0&0&0&\bar\gamma_{-2,-1}&\bar\gamma_{-2,-2}&\dots\\
\dots&\vdots&\vdots&\vdots&\vdots&\vdots&\ddots
\end{bmatrix}=
\begin{bmatrix}
  A&B\\0&D
\end{bmatrix}
.
\label{eq:68}
\end{equation}
It is convenient to think of these infinite block matrices as usual
scalar valued, infinite matrices, by identifying
$E_{bc}^{i,j}=E_{ij}\otimes E_{bc}\mapsto
\mathcal{E}_{b+ni,c+nj}$. This corresponds to identifying
$H^{(n)}=\mathbb{C}^{n}\otimes \mathbb{C}[z,z\inv]$ with
$H=\oplus_{s\in\mathbb{Z}}\mathbb{C}\epsilon_{s}$ via
$e_{a}z^{k}\mapsto \epsilon_{a+nk}$. This induces an identification of
the semi-infinite vacuum vector in $F^{(n)}$,
$v_{0}=e_{0}\wedge e_{1}\wedge \dots \wedge e_{n-1}\wedge
ze_{0}\wedge\dots$, with the vacuum vector
$w_{0}=\epsilon_{0}\wedge\epsilon_{1}\wedge\dots$ in $F^{(1)}$. Below, we will occasionally make these identifications and it will be clear when we are doing this based on whether we denote the vacuum vector by $v_0$ or $w_0$.

The $\tau$-functions are infinite determinants: if $g=g_{-}g_{0+}$, then
\[
\tau(g)=\langle v_{0}, gv_{0}\rangle=\langle w_{0},g_{0+}w_{0}\rangle=\det(A).
\]
Here (and below) we use that $g_{-}$ is of the form $1+X$, $X\colon
H^{(n)}_{+}\to H^{(n)}_{-}$, and multiplying by $g_{-}$ does not
change determinants.

There are $n$ distinct fermionic translation operators. They act on
the vacuum by
\begin{align*}
Q_{a}w_{0}&=\epsilon_{a-n}\wedge
w_{0}=\epsilon_{a-n}\wedge\epsilon_{0}\wedge\epsilon_{1}\wedge\dots,
\\
Q_{a}\inv w_{0}&=i(\epsilon_{a})w_{0}=
(-1)^{a}\epsilon_{0}\wedge\epsilon_{1}\wedge\dots\wedge\xcancel{\epsilon_{a}}\wedge\dots.
\end{align*}
Now
\[
\langle Q_{a}w_{0},gQ_{a^{\prime}}w_{0}\rangle=\langle
Q_{a}w_{0},g_{0+}Q_{a^{\prime}}w_{0}\rangle=
\bar\gamma^{-1,-1}_{aa^{\prime}}\det(A),
\]
and
\[
\bar\gamma^{-1,-1}_{aa^{\prime}}=
\frac{\langle Q_{a}w_{0},gQ_{a^{\prime}}w_{0}\rangle}
{\tau(g)}.
\]
In the same way we find that
\[
\bar\gamma^{-1,-1(\mathbf{k},\bm\beta)}_{aa^{\prime}}=
\frac{\langle Q_{a}w_{0},g^{(\mathbf{k},\bm\beta)}Q_{a^{\prime}}w_{0}\rangle}
{\tau(g^{(\mathbf{k},\bm\beta)})}.
\]

Next consider the inverse of the block $A$ in $g_{0+}$, again ignoring
for a moment the dependence on $(\mathbf{k},\bm\beta)$. The 
component $\tilde\gamma_{aa^{\prime}}^{00}$ of $g_{0+}\inv$  is of
course also the $(A\inv)_{aa^{\prime}}^{00}$ component of $A\inv$.
On the other hand 
\[
\left(A\inv\right)_{aa^{\prime}}^{00}=(-1)^{a+a^{\prime}}\det\left( A^{[a0]}_{[a^{\prime}0]}\right)/\det(A),
\]
where $A^{[a0]}_{[a^{\prime}0]}$ is the matrix obtained from $A$ by
deleting the row and column containing the $E_{aa^{\prime}}^{0,0}$ entry. Now
\[
(-1)^{a+a^{\prime}}\det(A^{[a0]}_{[a^{\prime}0]})=\langle Q_{a^{\prime}}\inv
v_{0},g_{0+}Q_{a}\inv v_{0}\rangle
\]
so that
\[
\tilde\gamma_{aa^{\prime}}^{00}=\frac{\langle Q_{a}Q_{a^{\prime}}\inv v_{0},Q_{a}gQ_{a}\inv v_{0}\rangle}{\tau(g)}
\]
In the same way we find
\[
\tilde\gamma_{aa^{\prime}}^{00(\mathbf{k},\bm\beta)}=\frac{\langle
  Q_{a}Q_{a^{\prime}}\inv v_{0},Q_{a}g^{(\mathbf{k},\bm\beta)}Q_{a}\inv
  v_{0}\rangle}
{\tau_{\mathbf{k}}^{(\bm\beta)}(g)}
\]
\end{proof}

To find \eqref{eq:65} we need to calculate a product of matrix
elements.

\begin{lem} \label{lem:QQtau}
  \begin{equation*}
    \langle Q_{a} v_{0},
    g^{(\mathbf{k},\bm\beta+\delta_{I})}Q_{a^{\prime\prime}} v_{0}
    \rangle \langle Q_{a^{\prime}}\inv v_{0},
    g^{(\mathbf{k},\bm\beta+\delta_{J})}Q_{a^{\prime\prime}}\inv v_{0}
    \rangle=\epsilon(\mathbf{k}+\delta_{a^{\prime\prime}},\delta_{a}-\delta_{a^{\prime}})
\tau_{\mathbf{k}+\delta_{a^{\prime\prime}}-\delta_{a}}^{(\bm\beta+\delta_{I}-\delta_{a^{\prime\prime}})}
\tau_{\mathbf{k}+\delta_{a^{\prime}}-\delta_{a^{\prime\prime}}}^{(\bm\beta+\delta_{J}+\delta_{a^\prime\prime})}.
\end{equation*}
\end{lem}

\begin{proof}
Let $A= \langle
  Q_{a} v_{0},
  g^{(\mathbf{k},\bm\beta+\delta_{I})}Q_{a^{\prime\prime}}  v_{0}
  \rangle 
\langle 
  Q_{a^{\prime}}\inv v_{0},
  g^{(\mathbf{k},\bm\beta+\delta_{J})}Q_{a^{\prime\prime}}\inv  v_{0}
  \rangle$. Then by unitarity \eqref{eq:21} and \eqref{eq:80}
\[
A=    \langle
    Q_{a^{\prime\prime}}\inv Q_{a} v_{0},
    Q_{a^{\prime\prime}}\inv\left(T^{\mathbf{k}}\right)\inv g^{(\bm\beta+\delta_{I})}Q_{a^{\prime\prime}} v_{0}
    \rangle
    \langle 
    Q_{a^{\prime\prime}} Q_{a^{\prime}}\inv v_{0},
    Q_{a^{\prime\prime}}\left(T^{\mathbf{k}}\right)\inv g^{(\bm\beta+\delta_{J})}Q_{a^{\prime\prime}}\inv  v_{0}
    \rangle
\]
 Using Lemma \ref{Lem:GeneratorsTranslationGroupIdentities}, Part (1)
 and the unitarity of $T^{\mathbf{k}}$ \eqref{eq:22},
\[
A=\langle
    T^{\mathbf{k}}Q_{a^{\prime\prime}}\inv Q_{a} v_{0},
    Q_{a^{\prime\prime}}\inv g^{(\bm\beta+\delta_{I})}Q_{a^{\prime\prime}} v_{0}
    \rangle
    \langle 
    T^{\mathbf{k}}Q_{a^{\prime\prime}} Q_{a^{\prime}}\inv v_{0},
    Q_{a^{\prime\prime}} g^{(\bm\beta+\delta_{J})}Q_{a^{\prime\prime}}\inv  v_{0}
    \rangle.
\]
By the definition $T^{\mathbf{k}}=Q^{-\mathbf{k}}$ and \eqref{eq:80}
\[
A=\langle
    Q^{-\mathbf{k}}Q_{a^{\prime\prime}}\inv Q_{a} v_{0},
     g^{(\bm\beta+\delta_{I}-\delta_{a^{\prime\prime}})} v_{0}
    \rangle
    \langle 
    Q^{-\mathbf{k}}Q_{a^{\prime\prime}} Q_{a^{\prime}}\inv v_{0},
     g^{(\bm\beta+\delta_{J}+\delta_{a^{\prime\prime}})}  v_{0}
    \rangle.
\]
Finally using Lemma \ref{lem:TensorQ} and Corollary
\ref{cor:tensorQQ}, and properties of the cocycle from Lemma \ref{lem:propcocycle}
\begin{multline*}
  A=\epsilon(\mathbf{k}+\delta_{a^{{\prime}\prime}},\delta_{a}-\delta_{a^{\prime}})
  \langle Q^{-\mathbf{k}-\delta_{a^{\prime\prime}}+\delta_{a}} v_{0},
  g^{(\bm\beta+\delta_{I}-\delta_{a^{\prime\prime}})} v_{0} \rangle
  \times \\
 \times  \langle Q^{-\mathbf{k}+\delta_{a^{\prime\prime}}-\delta_{a^{\prime}}}\
  v_{0}, g^{(\bm\beta+\delta_{J}+\delta_{a^{\prime\prime}})} v_{0}
  \rangle = \\
=\epsilon(\mathbf{k}+\delta_{a^{\prime\prime}},\delta_{a}-\delta_{a^{\prime}}) \langle T^{\mathbf{k}+\delta_{a^{\prime\prime}}-\delta_{a}} v_{0},
  g^{(\bm\beta+\delta_{I}-\delta_{a^{\prime\prime}})} v_{0} \rangle
  \langle T^{\mathbf{k}+\delta_{a^{\prime}}-\delta_{a^{\prime\prime}}}\ v_{0},
  g^{(\bm\beta+\delta_{J}+\delta_{a^{\prime\prime}})} v_{0} \rangle=\\=\epsilon(\mathbf{k}+\delta_{a^{\prime\prime}},\delta_a-\delta_{a^{\prime}})\tau_{{\mathbf{k}+\delta_{a^{\prime\prime}}-\delta_a}}^{(\bm\beta+\delta_{I}-\delta_{a^{\prime\prime}})}\tau_{{\mathbf{k}+\delta_{a^{\prime}}-\delta_a^{\prime\prime}}}^{(\bm\beta+\delta_{J}+\delta_{a^{\prime\prime}})}\end{multline*}
\end{proof}

 By combining \eqref{eq:64}, \eqref{eq:65},  the lemmas
 \ref{lem:g0+I}, \ref{lem:QQtau} (for $a=a^{\prime}$) and \eqref{eq:53} we see that calculating the
 coefficient of $E_{aa}^{-1,0}$ in two ways gives:
 \begin{equation*}
1+
\sum_{b\in J}
\frac{\tau_{\mathbf{k}+\delta_{b}-\delta_{a}}^{(\bm\beta+\delta_{I})}}
       {\tau_{\mathbf{k}}^{(\bm\beta+\delta_{I})}}
\frac{\tau_{\mathbf{k}+\delta_{a}-\delta_{b}}^{(\bm\beta+\delta_{J})}}
       {\tau_{\mathbf{k}}^{(\bm\beta+\delta_{J})}}
=\sum_{a^{\prime\prime}\in I}
\frac{\tau_{\mathbf{k}+\delta_{a^{\prime\prime}}-\delta_{a}}^{(\bm\beta+\delta_{I}-\delta_{a^{\prime\prime}})}}
       {\tau_{\mathbf{k}}^{(\bm\beta+\delta_{I})}}
\frac{\tau_{\mathbf{k}+\delta_{a}-\delta_{a^{\prime\prime}}}^{(\bm\beta+\delta_{J}+\delta_{a^{\prime\prime}})}}
      {\tau_{\mathbf{k}}^{(\bm\beta+\delta_{J})}}.
 \end{equation*}
Clearing the denominators leads to:
\\[.2cm]
\fbox{
  \begin{minipage}{0.97\linewidth}
    {
      \begin{thm}[Long Relations]\label{thm:long-relations-1}
      For all $(\mathbf{k},\bm\beta)\in A_{n-1}\times
      \mathbb{Z}^{n}$ and all proper nonempty subsets $I\subset
      \{0,1,\dots,n-1\}$ with complement $J$ and $a\in I$ the
      $nT$ $\tau$-functions satisfy:
      \begin{equation*}
        \tau_{\mathbf{k}}^{(\bm\beta+\delta_{I})}
        \tau_{\mathbf{k}}^{(\bm\beta+\delta_{J})}+
        \sum_{b\in J}
        \tau_{\mathbf{k}+\delta_{b}-\delta_{a}}^{(\bm\beta+\delta_{I})}
        \tau_{\mathbf{k}+\delta_{a}-\delta_{b}}^{(\bm\beta+\delta_{J})}=
        \sum_{a^{\prime\prime}\in I}
        \tau_{\mathbf{k}+\delta_{a^{\prime\prime}}-\delta_{a}}^{(\bm\beta+\delta_{I}-\delta_{a^{\prime\prime}})}
        \tau_{\mathbf{k}+\delta_{a}-\delta_{a\prime\prime}}^{(\bm\beta+\delta_{J}+\delta_{a^{\prime\prime}})}.
    \end{equation*}
  Here $\epsilon$ is the cocycle \eqref{eq:71} for $\mathbb{Z}^{n}$.
      \end{thm}
      }
    \end{minipage}
  }

\ \\[.2cm]

\section{$2T$ and $3T$ Relations.}
\label{sec:2t-relations}

For $n=2$ the $\tau$-functions are indexed by $(\mathbf{k},\bm\beta)\in
A_{1}\times \mathbb{Z}^{2}$ or equivalently the $\tau$-functions have
the form $\tau_{k}^{(\alpha,\beta)}$, where $(k,(\alpha,\beta))\in
\mathbb{Z}\otimes \mathbb{Z}^{2}$. In this case there are no short
relations. For the long relations we need to choose two nonempty
disjoint subsets $I,J\subset \{0,1\}$, so $I=\{0\},J=\{1\}$ or vice
versa. Either choice leads to the same equations:
\[
\tau_{k}^{(\alpha,\beta+1)}\tau_{k}^{(\alpha+1,\beta)}=
\tau_{k+1}^{(\alpha,\beta+1)}\tau_{k-1}^{(\alpha+1,\beta)}+
\tau_{k}^{(\alpha,\beta)}\tau_{k}^{(\alpha+1,\beta+1)}.
\]
For $n=3$ there are both short relations, of length 3 and long
relations, of length 4. In this case the $\tau$-function is of the
form $\tau_{k,\ell}^{(\alpha,\beta,\gamma)}$, with
$((k,\ell),(\alpha,\beta,\gamma))\in A_{2}\times \mathbb{Z}^{3}$. 

The short relations depend on the choice of a $k=1$ element subset $I$
of $\{0,1,2\}$, and a corresponding $\delta_J$. For example, if we choose $I=\{0\}$, then there are three choices for $\delta_J$:
\begin{enumerate}
\item $\delta_{J}=\delta_{0}\implies \rho=0$.
\item $\delta_{J}=\delta_{1}\implies \rho=\delta_{0}-\delta_{1}=\alpha_{1}$
\item $\delta_{J}=\delta_{2}\implies \rho=\delta_{0}-\delta_{2}=\alpha_{1}+\alpha_{2}$.
\end{enumerate}
Then we need to fix $b, c\notin I$, $b\ne c$. So there are many
possibilities. For instance, if $\rho=0$, $b=1,c=2$ Theorem
\ref{thm:short-relations} gives 
\[
  \tau_{k,\ell}^{(\alpha,\beta,\gamma)}\tau_{k,\ell-1}^{(\alpha+1,\beta,\gamma)}=
  \tau_{k,\ell-1}^{(\alpha,\beta,\gamma)}\tau_{k,\ell}^{(\alpha+1,\beta,\gamma)}+
  \tau_{k+1,\ell}^{(\alpha,\beta,\gamma)}\tau_{k-1,\ell-1}^{(\alpha+1,\beta,\gamma)}.
\]
The long relations depend on the choice of complementary subsets $I,J$
of
$\{0,1,2\}$, so there are 6 choices. For instance if we take $I=\{0\}$
the 4 term relation becomes 
\begin{equation*}
  \tau_{k,\ell}^{(\alpha+1,\beta,\gamma)}\tau_{k,\ell}^{(\alpha,\beta+1,\gamma+1)}=
  \tau_{k-1,\ell}^{(\alpha+1,\beta,\gamma)}\tau_{k+1,\ell}^{(\alpha,\beta+1,\gamma+1)}+
  \tau_{k-1,\ell-1}^{(\alpha+1,\beta,\gamma)}\tau_{k+1,\ell+1}^{(\alpha,\beta+1,\gamma+1)}+
  \tau_{k,\ell}^{(\alpha,\beta,\gamma)}\tau_{k,\ell}^{(\alpha+1,\beta+1,\gamma+1)}.
\end{equation*}

\section*{Acknowledgments}
The authors are grateful for travel support from the Simons Foundation, Collaboration Grant 245048. Addabbo expresses thanks for support from the Associate Alumnae of Douglass College and the Dr. Lois M. Lackner Mathematics Fellowship. 

\providecommand{\bysame}{\leavevmode\hbox to3em{\hrulefill}\thinspace}
\providecommand{\MR}{\relax\ifhmode\unskip\space\fi MR }

\providecommand{\MRhref}[2]{
  \href{http://www.ams.org/mathscinet-getitem?mr=#1}{#2}
}
\providecommand{\href}[2]{#2}

\end{document}